\documentclass[11pt,leqno]{amsart}

\usepackage{parskip}
\allowdisplaybreaks
\newcommand\del{\partial}
\newcommand\ext{\mathrm{d}}
\newcommand{\res}{\mathbin{\hspace{0.1em}\vrule height 1.3ex depth 0pt width 0.13ex\vrule height 0.13ex depth 0pt width 1.0ex}} 
\newcommand\sfrak{\mathfrak{s}}

\usepackage[left = 2.54cm, right = 2.54cm, top = 2.54cm, bottom = 2.54cm]{geometry}
\usepackage{amsmath}
\usepackage{amsfonts}
\usepackage{amssymb}
\usepackage{amsthm}
\usepackage{mathtools}
\usepackage{graphicx,caption, color, hyperref}
\usepackage{enumitem}
\usepackage{multicol}
\usepackage{bbm}
\usepackage{stmaryrd}
\usepackage{tikz}
\usepackage{enumitem}
\usepackage{bm}
\usepackage{pict2e}
\usepackage{subcaption}
\usepackage[mathscr]{euscript}

\usepackage{indentfirst}
\newlist{multienum}{enumerate}{1}
\setlist[multienum]{
    label=\alph*),
    before=\begin{multicols}{2},
    after=\end{multicols}
}
\usepackage{hyperref}
\usepackage[nameinlink]{cleveref}
\hypersetup{
    colorlinks=true,
    linkcolor=black,
    filecolor=black,      
    urlcolor=black,
    citecolor={black}
    }

\numberwithin{equation}{section}
\newcounter{theorem}

\newtheorem{prop}[theorem]{Proposition}
\newtheorem{lem}[theorem]{Lemma}
\newtheorem{theorem}[theorem]{Theorem}
\newtheorem{corollary}[theorem]{Corollary}

\newtheorem{lemma}[theorem]{Lemma}

\numberwithin{theorem}{section}

\theoremstyle{definition}
\newtheorem{defn}[theorem]{Definition}

\newtheoremstyle{newtheoremstyledefn}
{3pt}
{3pt}
{}
{\parindent}
{\bfseries}
{.}
{0.5em}
{} 

\theoremstyle{newtheoremstyledefn}
\newtheorem{remark}[theorem]{Remark}

\newcommand\A{\mathcal{A}}

\newcommand\FB{\mathfrak{B}}
\newcommand\C{\mathbb{C}}
\newcommand\CC{\mathcal{C}}

\newcommand\BC{\mathbf{C}}

\renewcommand\H{\mathcal{H}}
\newcommand\CK{\mathcal{K}}

\newcommand\R{\mathbb{R}}

\renewcommand\S{\mathcal{S}}

\newcommand\Z{\mathbb{Z}}

\newcommand{\reg}{\textnormal{reg}}
\newcommand{\sing}{\textnormal{sing}}

\newcommand{\al}{\alpha}
\renewcommand{\dim}{\textnormal{dim}}

\newcommand{\smallnorm}[1]{\lVert#1\rVert}

\newcommand{\dist}{\textnormal{dist}}

\newcommand{\sref}[2]{\hyperref[#2]{#1 \ref{#2}}}
\newcommand{\spt}{\textnormal{spt}}

\newcommand{\graph}{\textnormal{graph}}

\renewcommand{\div}{\textnormal{div}}

\newcommand{\mb}[1]{\llbracket{#1}\rrbracket}
\newcommand\w{\omega}
\newcommand\vphi{\varphi}
\renewcommand\phi{\vphi}
\newcommand\eps{\varepsilon}
\renewcommand{\bf}[1]{\mathbf{#1}}

\makeatletter

\usepackage{fancyhdr}
\author{Paul Minter \and Sidney Stanbury}

\address{\textnormal{Department of Pure Mathematics and Mathematical Statistics, University of Cambridge}}
\email{pdtwm2@cam.ac.uk, sas258@cam.ac.uk}

\title{Area minimising hypersurfaces mod $p$ do not admit immersed branch points}
\date{}
\begin{document}
\begin{abstract}
We show that area minimising hypersurfaces mod $p$ do not admit immersed branch points, namely branch points about which all classical singularities are immersed. Furthermore, we show that if an $n$-dimensional area minimising hypersurface mod $p$ is smoothly immersed outside an $\H^{n-1}$-null set, then it is in fact smoothly immersed outside a closed set of Hausdorff dimension at most $n-3$.

These results are consequences of a more general analysis of immersed stable minimal hypersurfaces with a certain `alternating' orientation. Indeed, our proof does not rely on the minimising property other than through stationarity, stability, and the verification of simple structural properties of the hypersurface.
\end{abstract}

\maketitle

\tableofcontents

\section{Introduction \& Main Results}\label{sec:introduction}

A classical result in geometric analysis is that an area minimising hypersurface $M^n$ (namely, a codimension one \emph{integer} multiplicity locally area minimising current) is smoothly embedded outside a (closed) singular set\footnote{For us, a \emph{singular point} will always refer to a point where the hypersurface is not smoothly \emph{embedded} on some neighbourhood of the point.} of Hausdorff dimension $\leq n-7$. Furthermore, under a mass bound, such hypersurfaces form a compact set in the topology of currents. A cornerstone of this theory is the $\eps$-regularity theorem of De Giorgi \cite{DG61} (cf.~\cite{Allard72}), a consequence of which is that area minimising hypersurfaces \emph{do not admit} flat singular points, i.e.~singular points with planar tangent cones (with multiplicity).

A related class of minimal hypersurfaces are those which are \emph{area minimising mod $p$} for some integer $p\in \{2,3,\dotsc\}$. They arise, for instance, as area minimising representatives of homology classes with $\Z_p\equiv \Z/p\Z$ coefficients. Their regularity theory is rather different to the area minimising case\footnote{Throughout this paper, by ``area minimising'' we will mean area minimising with $\Z$-coefficients, whilst we will always write ``area minimising mod $p$'' to refer to $\Z_p$-coefficients.} as, for example, they can have large (codimension $1$) singular sets. The primary structural difference lies in that area minimising hypersurfaces mod $p$ may exhibit \emph{classical singularities}\footnote{A \emph{classical singularity} is a singular point where locally the hypersurface is given by the sum of a finite number of $C^{1,\alpha}$ submanifolds-with-boundary, all of whose boundaries coincide. These include, for instance, triple junction singularities and immersed singular points where individual ($n$-dimensional) sheets meet transversely along a common $(n-1)$-dimensional submanifold.} (see Figure \ref{fig:mod-p-classical-singularity}).

\begin{figure}[h]
	\centering
	\begin{tikzpicture}
		\draw (0,0) circle (1.5);
		\filldraw (0,1.5) circle (0.05);
		\filldraw (0.9,1.2) circle (0.05);
		\filldraw (-0.8,-1.275) circle (0.05);
		\filldraw (1.465,-0.3) circle (0.05);
		\filldraw (-1.475, 0.28) circle (0.05);
		\filldraw (-0.2,-1.485) circle (0.05);
		\filldraw (-0.2,0.1) circle (0.05);
		
		\draw (-0.2,0.1) -- (0,1.5);
		\draw (-0.2,0.1) -- (0.9,1.2);
		\draw (-0.2,0.1) -- (-0.8,-1.275);
		\draw (-0.2,0.1) -- (1.465,-0.3);
		\draw (-0.2,0.1) -- (-1.475,0.28);
		\draw (-0.2,0.1) -- (-0.2,-1.485);
	\end{tikzpicture}
	\caption{\footnotesize If one takes $p\geq 3$ distinct points (positively oriented) on a circle in $\R^2$, an area minimiser mod $p$ with this boundary will always contain a classical singularity (of density $p/2$) in the interior. If $p\geq 6$ is even and the points are in general position, there will always be a non-immersed classical singularity (as generically $p/2$ ($\geq 3$) lines do not intersect at a common point).}
	\label{fig:mod-p-classical-singularity}
\end{figure}
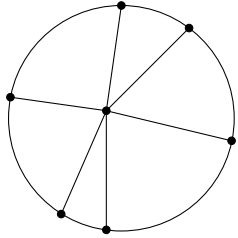

The presence of classical singularities makes the theory closer to that of minimal \emph{immersions}, rather than minimal \emph{embeddings} (as in the area minimising case). A key difficulty is that one cannot reduce to a ``multiplicity one'' setting, and is forced to consider higher multiplicity situations. However, when classical singularities are absent, the theory of area minimising hypersurfaces mod $p$ mirrors that of area minimising hypersurfaces; indeed, both settings are then contained within the stability theory of Wickramsekera \cite{Wic14} (see \cite{MW24} for further discussion).

Area minimising hypersurfaces mod $p$ can also have flat singular points (see Figure \ref{fig:flat-example}). In turn, one is forced to consider the possibility of \emph{branch points}, i.e.~flat singular points about which the hypersurface does not decompose as a finite sum of minimal graphs over the tangent plane. One expects the structure of the hypersurface about a typical branch point to be modelled on $\text{Re}(z^{r/s})$ for coprime integers $r> s\geq 2$, leading to a loss of analyticity (and, if $r<2s$, curvature control). Branch points cause significant complications for compactness and regularity theories.

\vspace{-1.5em}

\begin{figure}[h]
	\centering
\tikzset{every picture/.style={line width=0.75pt}}       

\begin{tikzpicture}[x=0.75pt,y=0.75pt,yscale=-1,xscale=1]
\draw    (285.44,8.8) .. controls (161.86,140.04) and (318.68,169) .. (344.25,131.89);
\draw    (344.25,131.89) .. controls (357.88,93.88) and (284.59,-20.16) .. (353.62,25.1);
\draw    (344.25,131.89) .. controls (352.77,125.56) and (364.7,119.22) .. (370.67,92.07);
\draw    (342.54,95.69) .. controls (269.25,101.12) and (189.13,92.07) .. (241.12,75.05);
\draw    (342.54,95.69) .. controls (420.95,85.74) and (386.01,74.88) .. (372.37,71.25);
\draw    (286.01,67.94) .. controls (297.94,67.03) and (311,67.8) .. (323.51,68.54);
\draw  [dash pattern={on 0.84pt off 2.51pt}]  (241.12,75.05) .. controls (257.31,70.35) and (269.53,68.54) .. (286.01,67.94);
\draw    (353.62,25.1) .. controls (370.67,39.58) and (374.08,61.3) .. (372.37,71.25);
\draw  [dash pattern={on 0.84pt off 2.51pt}]  (265.84,70.35) -- (302.06,81.66);
\draw    (302.06,81.66) -- (342.54,95.69);
\draw    (238.14,92.3) -- (305.33,82.42);
\draw    (285.44,8.8) .. controls (277.77,51.8) and (283.73,80.31) .. (302.06,81.66);
\draw    (302.06,81.66) .. controls (315.27,83.02) and (330.18,72.61) .. (325.5,37.31);
\draw  [dash pattern={on 0.84pt off 2.51pt}]  (305.33,82.42) -- (336.58,76.38);
\draw    (336.58,76.38) -- (372.37,71.25);
\draw  [dash pattern={on 0.84pt off 2.51pt}]  (323.51,68.54) .. controls (336.33,69.14) and (358.33,69.14) .. (372.37,71.25);
\end{tikzpicture}
\vspace{-1.5em}
\caption{\footnotesize An area minimiser mod $4$ with a flat singular point (the tangent plane has multiplicity $2$). Note that this is not a branch point, as the surface decomposes as Enneper's surface with a tangent plane. The fact that this surface is area minimising mod $4$ follows from \cite{White79}.}
\label{fig:flat-example}
\end{figure}
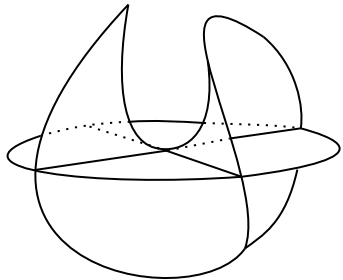

The aim of the present paper is to address some of these questions concerning the structure and regularity of area minimising hypersurfaces mod $p$ and to demonstrate that it is better than one might initially expect (cf.~\cite[Problem 4.20]{GMT}). The key insight underpinning much of our work is that, in analogy with area minimising hypersurfaces, ``genuine'' branch points do not appear to be present in area minimising hypersurfaces mod $p$. Consequently, we wish to understand when an area minimiser mod $p$ is locally a finite sum of embedded minimal hypersurfaces (with empty interior boundary), up to generally unavoidable singularities.

Our first result in this direction is that area minimising hypersurfaces mod $p$ do not admit branch points when all their classical singularities are immersed\footnote{By this, we mean that the tangent cone $\BC$ to \emph{any} classical singularity is the sum of a finite number of \emph{hyperplanes}, rather than only half-hyperplanes, i.e.~$\BC = \sum_i |P_i|$ for hyperplanes $P_i$, with at least two being distinct, and, for any $i,j$ with $P_i,P_j$ distinct, $P_i\cap P_j$ is a fixed $(n-1)$-dimensional subspace (independent of $i,j$). Equivalently, the local structure of $V$ about any classical singularity is given by the sum of (intersecting, $C^1$) embedded submanifolds, rather than only submanifolds with (interior) boundary. Under minimality, the $C^1$ regularity is then upgradable to $C^\infty$.}:

\begin{theorem}\label{thm:main-no-bp}
    Let $T^n$ be an area minimising hypersurface mod $p$ in $B^{n+1}_2(0)$ with $\del^p T = 0$. Suppose all classical singularities of $T$ are immersed. Then, $T$ does not have any branch points.
\end{theorem}

Consequently, when all classical singularities are immersed, the local structure about a flat singular point must be described by a finite union of minimal graphs over the (unique) tangent space. This extends a result of White from $p=4$ to general $p$ \cite{White79}. We note that when $p$ is odd or $p=2$ there are no flat singular points at all due to \cite{White84} and \cite{Allard72}, respectively.

In particular, if an area minimising hypersurface mod $p$ admits a branch point, then the branch point must be a limit point of \emph{non-immersed} classical singularities. As far as the authors are aware, such an example has yet to be constructed. Nonetheless, even though the existence of such an example seems plausible, we do not believe it to have a ``genuine'' branch point. For instance, if instead of assuming that all classical singularities of the current are immersed, we only assume that all classical singularities of the support are immersed, then the support does not have any branch points. We discuss this further in Section \ref{sec:rest}.

We stress that Theorem \ref{thm:main-no-bp} is strictly a codimension one phenomenon. In codimension $>1$ area minimisers mod $p$ can have branch points. This is analogous to (and follows from) how area minimisers in codimension $>1$ exhibit branch points, as seen through complex subvarities of $\C^n$.

Building on Theorem \ref{thm:main-no-bp}, we establish the following (stronger) partial regularity theorem:

\begin{theorem}\label{thm:main-reg}
    Let $T^n$ be an area minimising hypersurface mod $p$ in $B^{n+1}_2(0)$ with $\del^p T = 0$. Suppose all classical singularities of $T$ are immersed. Then, $T$ is represented by a two-sided smooth immersion outside a (closed) set of Hausdorff dimension at most $n-3$.
\end{theorem}

We do not expect the $n-3$ bound in Theorem \ref{thm:main-reg} to be optimal; we suspect the sharp bound should be $n-7$ in line with known singularity models for area minimising currents. The obstruction to improving $n-3$ to $n-7$ lies in that there could be non-flat hypercones in $\R^{n+1}$ for $4\leq n+1\leq 7$ which are area minimising mod $p$ and represented by two-sided smooth immersions away from $0$ (see also Remark \ref{remark:ambient-stable}). We note that when $p=2,4$ or $p$ is odd no such cones exist, and so the conclusion of Theorem \ref{thm:main-reg} improves in these cases to $n-7$.

Notice that Theorem \ref{thm:main-reg} shows, for $T$ an $n$-dimensional area minimising hypersurface mod $p$, the following are equivalent:
\begin{enumerate}
    \item [(a)] All classical singularities of $T$ are immersed;
    \item [(b)] $T$ is smoothly immersed outside a closed set of Hausdorff dimension $\leq n-3$;
    \item [(c)] $T$ is smoothly immersed outside an $\H^{n-1}$-null set;
\end{enumerate}
Clearly $\text{(b)}\Rightarrow\text{(c)}\Rightarrow\text{(a)}$. The implication $\text{(a)}\Rightarrow \text{(b)}$ is non-trivial, and is exactly Theorem \ref{thm:main-reg}. In particular, having only classical singularities which are immersed is a necessary and sufficient condition for an area minimising hypersurface mod $p$ to be smoothly immersed away from a (closed) set of codimension $\geq 3$.

Our proof of Theorem \ref{thm:main-no-bp} and Theorem \ref{thm:main-reg} utilises an $\eps$-regularity theorem for immersed area minimising hypersurfaces mod $p$ (see Theorem \ref{thm:eps-reg}). This can be viewed as giving an $\eps$-regularity theorem akin to De Giorgi's $\eps$-regularity theorem for area minimising currents. However, we stress that its proof is substantially more complicated, as unlike in the area minimising case one does not appear to be able to reduce to a multiplicity one situation. Indeed, the starting point of our proof is the (possibly branched) $\eps$-regularity theorem of the first author with Wickramasekera \cite{MW24}.

This $\eps$-regularity theorem allows us to also prove the following compactness theorem:

\begin{theorem}\label{thm:main-compactness}
    Let $(T_j)_j$ be a sequence of area minimising hypersurfaces mod $p$ in $B^{n+1}_2(0)$ with $\del ^p T_j = 0$. Suppose that:
    \begin{itemize}
    	\item all classical singularities of each $T_j$ are immersed;
	\item $\limsup_j\|T_j\|(B^{n+1}_2(0))<\infty$.
    \end{itemize}
    Then, there exists a subsequence $\{j^\prime\}$ and $T$ an area minimising hypersurface mod $p$ in $B^{n+1}_2(0)$ with $\del^pT=0$ which is represented by a two-sided smooth immersion outside a closed set $S$ with $\dim_\H(S)\leq n-3$ such that $T_{j^\prime}\rightharpoonup T$ as currents and $T_{j^\prime}\to T$ locally smoothly as immersions in $B^{n+1}_2(0)\setminus S$.
\end{theorem}

The $n-3$ bound in Theorem \ref{thm:main-compactness} comes from the dimension bound in Theorem \ref{thm:main-reg}.

When $p$ is odd or $p=2,4$, the above results already follow from known results -- we discuss this and other relevant background in Section \ref{sec:background}. As such, we are primarily interested in the case of even $p\geq 6$.

In fact, all the above results for area minimisers mod $p$ are special cases of more general results we prove for stable minimal hypersurfaces which obey various structural conditions. As such, our proofs \emph{do not} utilise the minimising mod $p$ assumption other than through the verification of these structural conditions, which is rather simple to do. More precisely, the only structural conditions we need in addition to the variational assumptions (stationarity and stability of the smoothly embedded part of the hypersurface) are:
\begin{enumerate}
    \item [(i)] All classical singularities and flat singular points have density $p/2$;
    \item [(ii)] The hypersurface has a special type of `alternating' orientation (see Definition \ref{defn:orientation}).
\end{enumerate}
The philosophy that the regularity theory of area minimising hypersurfaces mod $p$ is subsumed by that of stable minimal hypersurfaces is one which was utilised in the work \cite{MW24} of the first author and Wickramasekera. It can be compared with how the work of Wickramasekera \cite{Wic14} shows that the regularity theory for area minimising hypersurfaces is contained within that of stable minimal hypersurfaces. These more general results are stated in Section \ref{sec:background} and Appendix \ref{app:A}.

\begin{remark}\label{remark:ambient-stable}
Our results show a closeness between area minimising hypersurfaces mod $p$ and two-sided stable minimal immersions. However, we stress that this stability is only \emph{ambient} stability (which in particular implies stability of the regular part), rather than stability \emph{as an immersion}. This difference can be seen in that immersed classical singularities are comprised of \emph{exactly} $p/2$ embedded submanifolds \emph{all} intersecting along a \emph{common} $(n-1)$-dimensional submanifold. Competitors must preserve this structure (to avoid creating an interior boundary mod $p$), whereas generic individual (normal) perturbations of $\geq 3$ hypersurfaces would only intersect along a common submanifold of dimension at most $n-2$. Thus, when $p/2\geq 3$, i.e.~$p\geq 6$, one does not expect stability of the immersion (and one does when $p=4$). This also explains why one cannot hope to prove (easily at least) for $p\geq 6$ an analogous result to White's \cite{White79} regarding the local decomposability of mod $4$ minimising hypersurfaces into mod $2$ minimising hypersurfaces (in the smoothly immersed case) as such a result would imply stability of the immersion (although, as noted this is true when $p=4$). Thus, despite Theorem \ref{thm:main-reg}, the theory of immersed area minimisers mod $p$ for $p\geq 6$ does not appear to be contained within that of stable minimal immersions, and so one cannot use that theory (e.g.~\cite{Bel25}) to immediately improve upon the results in Theorem \ref{thm:main-reg} and Theorem \ref{thm:main-compactness}. For instance, a Simons' classification for minimal hypercones in $\R^{n+1}$ for $4\leq n+1\leq 7$ which are represented by two-sided smooth immersions outside of $0$ and are ambiently stable does not appear to be known.
\end{remark}

\begin{remark}[The Riemannian Case]
Even though the above results are stated locally in Euclidean space, they have analogues in Riemannian manifolds. Being local statements, Theorem \ref{thm:main-no-bp} and Theorem \ref{thm:main-reg} are unchanged, except for the two-sided claim in Theorem \ref{thm:main-reg}. This depends on the ambient space and whether the current has constant multiplicity $p/2$ on its regular part (which is always true when $p=2$ and impossible if $p$ is odd). For instance, if the regular part has multiplicity $<p/2$ the mod $p$ structure still leads to a continuous well-defined unit normal to $T$, but this is no longer true if the multiplicity is $p/2$ (for instance, $\R\text{P}^2\subset \R\text{P}^3$ is area minimising mod $2$). For Theorem \ref{thm:main-compactness}, the same modification concerning two-sidedness is needed, but the other conclusions remain.
\end{remark}

\textbf{Overview of the Proof.} The main technical result we need is the $\eps$-regularity theorem (Theorem \ref{thm:eps-reg}). For this, our starting point is the $\eps$-regularity theorem established in \cite[Theorem A]{MW24}, which gives that when the hypersurface is close to a multiplicity $Q$ ($=p/2$) plane, its structure is described by the graph of a  \emph{multi-valued} function $u:B^n_1(0)\to \A_Q(\R)$. Our aim is then to show that our assumptions imply that this graph decomposes, on all of $B^n_1(0)$, as a sum of (smooth) single-valued minimal graphs.

First, we use the structural conditions to get that the multi-valued graph can be written as a sum of $C^{1,\alpha}$ single-valued functions $u_1,\dotsc,u_Q:B^n_1(0)\to \R$. This is the primary place we use our structural conditions. We then need to show that these $u_i$ solve the minimal surface equation (MSE). We first show that they solve the MSE away from the (closed) $Q$\emph{-critical set}:
$$\mathcal{K} := \{x\in B^n_1(0): u_i(x) = u_j(x) \text{ and }Du_i(x) = Du_j(x)\text{ for all }i, j\}.$$
Of course, if $u_i\equiv u_j$ for all $i, j$ then the hypersurface is just a single minimal graph with multiplicity, so we can assume $\mathcal{K}\neq B^n_1(0)$. We want to show that $u_i$ also solves the MSE across $\mathcal{K}$. As $\mathcal{K}$ is not necessarily a level set of $u_i$, for it to be removable from the PDE we want to show that it is sufficiently small. Knowing $\H^{n-1}(\mathcal{K})=0$ would suffice. It has recently been established in \cite{DLHMS+22,KMW} that $\dim_\H(\mathcal{K})\leq n-2$, and so one could simply apply this result to conclude. However, in the present setting we are instead able to give a much simpler proof of this bound following more classical PDE methods which we believe to be of independent interest. In particular, our proof \emph{does not} use Almgren's center manifold in any form (we stress that both \cite{DLHMS+22, KMW} utilise a center manifold).

To describe our approach, we look at $v:=u_i-u_j$ for some $u_i\neq u_j$. We then have
$$\mathcal{K} \subseteq \{v=0, Dv=0\}$$
(in fact, in our setting this is an equality). The advantage of working with $v$ is that since $\mathcal{K}$ is now contained within a level set of $v$, one has hope to extend the PDE which $v$ satisfies across $\mathcal{K}$ to all of $B^n_1(0)$, and thus one can hope to show that $\mathcal{K}$ is small by using frequency function techniques for $v$. The two issues are: (i) $v$ is only $C^{1,\alpha}$; and (ii) the coefficients in the PDE satisfied by $v$ are only $C^{0,\alpha}$. One needs $\alpha=1$, so that the $u_i$ are $C^{1,1}$, in order to make progress. We achieve this by using the planar frequency function for the original $Q$-valued stationary graph $u$ (cf.~\cite{KMW}) combined with the observation that tangent maps must be homogeneous harmonic polynomials. Our proof of this latter fact in turn uses that a $C^1$ function which is harmonic away from its critical set must necessarily be harmonic across its critical set. We note that whilst the planar frequency function was first introduced in the work of Krummel--Wickramasekera for area minimising currents \cite{KW23b}, our use here, like in \cite{KMW}, is for a special class of \emph{stationary} objects. Once the $C^{1,1}$ regularity of $u_i$ has been established, one can not only extend the PDE satisfied by $v$ to being weakly satisfied on all of $B^n_1(0)$, but one has sufficient regularity of the coefficients in the PDE in order to prove monotonicity of the frequency function for $v$ as in Garofalo--Lin \cite{GarofaloLin86}. This allows us to prove $\dim_\H(\mathcal{K})\leq n-2$, and hence complete the proof of the $\eps$-regularity result.

\begin{remark}\label{remark:decomp}
The situation which arises in the above argument raises the following question: if $u_1,\dotsc,u_Q:B^n_1(0)\to \R^k$ are $C^{1,\alpha}$ with $V = \sum^Q_{i=1}\mathbf{v}(u_i)$ being stationary, must each $\mathbf{v}(u_i)$ be stationary?\footnote{Recall that $\mathbf{v}(u_i)$ denotes the (multiplicity one) varifold associated to the graph of $u_i$.} One difficulty lies in not knowing the size of the $Q$-touching set. The argument sketched above answers a special case of this. We suspect that this more general question also has a positive answer, although we do not pursue this here. Note that this question is false if each $u_i$ is only Lipschitz (for example, take the union of the graphs of $u_\pm(x) := \pm |x|$ for $x\in \R$). We believe it should also be true if each $u_i$ is $C^1$; certainly in the $Q=2$ one may first apply the results in \cite{BKMW25} to upgrade the regularity of each $u_i$ to $C^{1,\alpha}$, thus reducing to the $C^{1,\alpha}$ case.
\end{remark}

\textbf{Structure of the Paper.} For the remainder of this introduction we discuss the background of the problem, both for area minimising hypersurfaces mod $p$ and, more generally, stable minimal hypersurfaces. We also state our main $\eps$-regularity result (Theorem \ref{thm:eps-reg}). We then prove our main $\eps$-regularity theorem in Section \ref{sec:proof}, and subsequently prove all other results stated in the introduction in Section \ref{sec:other}. In Section \ref{sec:rest} we discuss what can be said in the general situation where one allows non-immersed classical singularities, and pose several open questions. Finally, in Appendix \ref{app:A} we prove generalisations of our results in the stable setting (where branch points can occur).

\textbf{Acknowledgements:} This research was conducted during the period PM served as a Clay Research Fellow. The research of SS was supported by the Engineering and Physical Sciences Research Council (EPSRC) - EP/W524633/1. The authors would like to thank Mattia Luchese and Neshan Wickramasekera for comments on an earlier version of the manuscript.

\subsection{Background}\label{sec:background}

Area minimising hypersurfaces mod $p$ are the natural objects minimising area when one works with the coefficient group $\Z_p$ rather than $\Z$. Some key features for small $p$ are:
\begin{itemize}
	\item area minimisers mod $2$: allow unoriented surfaces (e.g.~$\R \text{P}^2\subset \R \text{P}^3$);
	\item area minimisers mod $3$: can admit triple junction singularities;
	\item area minimisers mod $4$: locally decompose as the sum of area minimisers mod $2$ (\cite{White79}).
\end{itemize}
In particular, the latter result of White shows that area minimising hypersurfaces mod $4$ do not have branch points. The present work recovers this corollary without using the minimising assumption, other than through the verification of comparatively simple structural properties\footnote{This follows from Theorem \ref{thm:main-no-bp} (cf.~Theorem \ref{thm:eps-reg}), as all classical singularities are immersed in area minimising hypersurfaces mod $4$.}.

Singularities within area minimising hypersurfaces mod $p$ are rather constrained. For instance, it follows from \cite[Proposition 3.4]{DLHMSS21} and a simple competitor argument that any flat point of an area minimising hypersurface mod $p$ must have density at most $p/2$. Consequently, the density of an area minimising hypersurface mod $p$ is $\leq p/2$ almost everywhere. In fact, one has:

\begin{enumerate}
	\item [(a)] Any flat singular point has density \emph{exactly} $p/2$ (indeed, there is an $\eps$-regularity theorem near hyperplanes of multiplicity $<p/2$ \cite{White84}). In particular, if $p$ is odd or $p=2$ then $T$ has no flat singular points (the $p=2$ case follows from Allard regularity \cite{Allard72}).
	\item [(b)] Any classical singularity has density \emph{exactly} $p/2$. Furthermore, each (oriented) hypersurface in the classical singularity induces the same Stokes' orientation at the common boundary (see \cite[Proposition 3.5]{DLHMSS21}).
\end{enumerate}
These insights have led to advances in the regularity theory. Indeed, \cite{DLHMSS21} proves an $\eps$-regularity theorem near any classical cone (of density $p/2$). It is for these reasons why Theorem \ref{thm:main-no-bp}, Theorem \ref{thm:main-reg}, and Theorem \ref{thm:main-compactness} already follow from known results when $p=2,4$ or $p$ is odd. Furthermore, \cite{DLHMS+22} proves that the flat singular set has Hausdorff dimension $\leq n-2$.

In fact, these regularity results are special cases of results which hold for stable minimal hypersurfaces. Indeed, the emerging philosophy is that stationarity and stability (of the regular part) combined with restrictions on the possible classical singularities leads to strong regularity conclusions in codimension one. As mentioned in (b) above, we know (from simple $1$-dimensional comparison arguments) that area minimisers mod $p$ only contain classical singularities of density exactly $p/2$. It is therefore natural to study stable minimal hypersurfaces which only contain classical singularities of a fixed density. In this direction:
\begin{enumerate}
	\item [(A)] The work of Wickramasekera \cite{Wic14} implies the $\eps$-regularity theorems of \cite{DLHMSS21} and \cite{White84} (cf.~\cite[Theorems B and C]{MW24}, and Lemma \ref{lemma:alternating-check} for constancy of the orientation);
	\item [(B)] The work of the first author and Wickramasekera \cite{MW24} proves an $\eps$-regularity theorem, as a multi-valued graph, for stable minimal hypersurfaces near hyperplanes of multiplicity $Q$ provided there are no classical singularities of density $<Q$ (which applies to area minimising hypersurfaces mod $p$);
	\item [(C)] The work of Krummel, the first author, and Wickramasekera \cite{KMW} proves a general Hausdorff dimension bound on the singular set for multi-valued minimal graphs satisfying an $\eps$-regularity property. In particular, by (B) above, this applies to area minimising hypersurfaces mod $p$ and recovers the dimension bound in \cite{DLHMS+22}.
\end{enumerate}
The present work builds upon these links between stable minimal hypersurfaces and area minimisers mod $p$ by now making use of the special type of orientation that an area minimiser mod $p$ has.

First, we recall the definition of $\S_Q$ from \cite{MW24}\footnote{Throughout, we refer the reader to \cite{MW24} for more description of the notation and language where necessary.}. The class $\S_Q$ is the set of all integral varifolds $V$ in $B^{n+1}_2(0)$ with $0\in \spt\|V\|$, $\|V\|(B^{n+1}_2(0))<\infty$, and which obey:
\begin{enumerate}
	\item [$(\S1)$\ \ ] $V$ is stationary in $B^{n+1}_2(0)$;
	\item [$(\S2)$\ \ ] $\reg(V)$ is stable on any open $U\subseteq \{\Theta_V<Q+1/2\}$ for which $\dim_\H(\sing(V)\cap U)\leq n-7$;\footnote{In \cite{MW24}, the stability condition is required on any open $U\subset B^{n+1}_2(0)$ for which $\dim_\H(\sing(V)\cap U)\leq n-7$, i.e.~there is no additional requirement that $U\subseteq\{\Theta_V<Q+1/2\}$. However, this additional requirement does not impact any of the arguments in \cite{MW24}, and here it is slightly more convenient to work with this weaker assumption.}
	\item [$(\S3)_Q$] $V$ does not contain any classical singularities of density $<Q$.
\end{enumerate}
Note that if $T$ is area minimising mod $p$, then the natural varifold associated to $T$ belongs to $\S_{p/2}$. In fact, this varifold lies in a subclass of $\S_{p/2}$ of varifolds whose regular part admits an orientation with a certain `alternating' property which induces a bipartite structure on the connected components of $\reg(T)$ with multiplicity $<p/2$. This is a global structural property satisfied by area minimisers mod $p$ in addition to the (inherently local) property $(\S3)_{p/2}$ which we make critical use of.
\vspace{-0.5em}
\begin{defn}\label{defn:orientation}
We say that $V\in \S_Q$ is \emph{alternating} if $\reg(V)\cap\{\Theta_V<Q\}$ has an orientation which is continuous on the connected components of $\reg(V)\cap\{\Theta_V<Q\}$ with the following property: if $x$ is a density $Q$ classical singularity of $V$ and $\{M_i\}_{i=1}^{2Q}$ are the oriented submanifolds-with-boundary whose union is equal to the support of $V$ in a neighbourhood of $x$, then the Stokes' orientation on the (interior) boundaries of each $M_i$ all coincide. 
\end{defn}
\vspace{-0.5em}
We now prove that the natural varifold associated to an area minimising hypersurface mod $p$ is alternating.

\begin{lemma}\label{lemma:alternating-check}
Suppose $T$ is an area minimising hypersurface mod $p$ in $B^{n+1}_2(0)$, with $\partial^pT=0$. Let $V$ be the natural varifold associated to $T$. Then $V\in \S_{p/2}$ is alternating.
\end{lemma}
\begin{proof}
The main point is to check that $T$ has no interior \emph{integral} boundary on a region where $V$ is represented by a smooth (minimal) hypersurface $M$ of constant multiplicity $k<p/2$. By the constancy theorem for rectifiable currents \cite[4.1.7]{Fed69}, this then implies that the orientation which $T$ induces on $M$ is continuous (in fact, the same is also true on $\{\Theta_V<p/2\}$, since $\dim_\H(\sing(V)\cap \{\Theta_V<p/2\})\leq n-7$). The result then follows from the description of classical singularities established in \cite[Section 3]{DLHMSS21}. We note that one could refer to \cite[Theorem 3.1]{White84} for this, but the present situation is much simpler as we are assuming regularity a priori.

Suppose $\del T\res M \neq 0$, and let $B$ be an open ball centred on a point in $\spt\|\del T\|\cap M$ where $\del T \res B \neq 0$. Fix an orientation of $M$ in $B$ (which is possible due to $M$ being an embedded in $B$). Let $E$ be the set of points in $M\cap B$ where the orientation of $T$ agrees with the chosen orientation on $M$. By the constancy theorem, $T$ has a continuous orientation in $M\setminus \spt\|\del T\|$, and so $T\res B$ decomposes as
$$T\res B = 2k\llbracket M\cap B\rrbracket \res E - k\llbracket M\cap B\rrbracket.$$
Since $M\cap B$ has no interior boundary, $E$ is a Caccioppoli set, and thus $\|\del T\|\res B = 2k\cdot\H^{n-1}\res \del^*E$ (cf.~\cite[Section 14]{Simon83Lectures}). As $\del^*E$ has unit $\H^{n-1}$-density everywhere this implies that $\del T\res B$ has density $2k$ at $\H^{n-1}$-a.e.~point of $\spt\|\del T\|\cap B$. However, we need $\del T = 0$ mod $p$ which, since $2k<p$, is only possible if $\H^{n-1}\res \del^*E = 0$, and hence $\del T\res B = 0$, which is a contradiction.
\end{proof}

We digress to make some further remarks on Definition \ref{defn:orientation}:
\begin{enumerate}
\item [(a)] If $T$ is an area minimising hypersurface mod $p$, one cannot infer anything about the induced orientation on a regular component of multiplicity $p/2$, as one may arbitrarily change the unit normal without creating an interior boundary mod $p$. We stress that a significant difficulty in the present work is the possibility that a connected component of $\reg(V)$ with multiplicity $p/2$ could be bounded entirely by branch points. One morally needs a dimension bound on the branch set to say this cannot happen.

\item [(b)] As we are in codimension one, an orientation of $\reg(V)$ is a choice of unit normal. Being alternating is equivalent to a continuous choice of unit normal (on $\reg(V)$) such that the unit conormals to the boundary of a classical singularity must either all be oriented towards the boundary or all away from the boundary. As such, one may call a classical singularity a \emph{source} (respectively \emph{sink}) when all conormals are oriented \emph{away} (respectively \emph{towards}) from the classical singularity (see Figure \ref{fig:source/sink}).

Moreover, when the classical singularity is \emph{immersed}, the alternating condition simply means that the unit normal along each (embedded) submanifold of the classical singularity changes sign across the classical singularity.

\begin{figure}[h]
	\centering
	\begin{tikzpicture}
		\filldraw (0,0) circle (0.05);
		\draw [thick]  (0,0) -- (1,0.15);
		\draw [thick, <-] (1,0.15) -- (2,0.3);
		\draw [thick] (0,0) -- (-1,0.15);
		\draw [thick, <-] (-1,0.15) -- (-2,0.3);
		\draw [thick] (0,0) -- (0.5,0.7);
		\draw [thick, <-] (0.5,0.7) -- (1,1.4);
		\draw [thick] (0,0) -- (0.25,-0.8);
		\draw [thick, <-] (0.25,-0.8) -- (0.5,-1.6);
		\draw [thick] (0,0) -- (-0.9,0.35);
		\draw [thick, <-] (-0.9,0.35) -- (-1.8,0.7);
		\draw [thick] (0,0) -- (-0.6, -0.55);
		\draw [thick, <-] (-0.6,-0.55) -- (-1.2,-1.1);
		
		\filldraw (7,0) circle (0.05);
		\draw [thick, ->]  (7,0) -- (8,0.15);
		\draw [thick] (8,0.15) -- (9,0.3);
		\draw [thick, ->] (7,0) -- (6,0.15);
		\draw [thick] (6,0.15) -- (5,0.3);
		\draw [thick, ->] (7,0) -- (7.5,0.7);
		\draw [thick] (7.5,0.7) -- (8,1.4);
		\draw [thick, ->] (7,0) -- (7.25,-0.8);
		\draw [thick] (7.25,-0.8) -- (7.5,-1.6);
		\draw [thick, ->] (7,0) -- (6.1,0.35);
		\draw [thick] (6.1,0.35) -- (5.2,0.7);
		\draw [thick, ->] (7,0) -- (6.4, -0.55);
		\draw [thick] (6.4,-0.55) -- (5.8,-1.1);
	\end{tikzpicture}
	\caption{\footnotesize A classical singularity with a sink orientation (left) and a source orientation (right).}
	\label{fig:source/sink}
\end{figure}
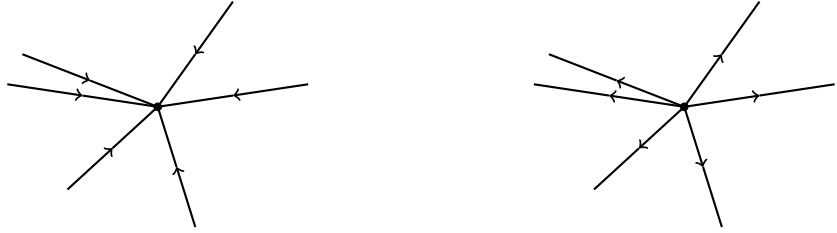

	\item [(c)] In a domain $\Omega$ where $V\in\S_Q$ only has regular points and immersed classical singularities of density $Q$, being alternating forces that $V\res\Omega$ is represented by a \emph{two-sided} smooth immersion. Indeed, the local structural conditions imply that $V\res\Omega$ is represented by the image of a smooth immersion $\iota:M\to B^{n+1}_2(0)$. To see that $M$ is two-sided, notice that (the local image of) $M$ has a well-defined continuous unit normal due to the alternating condition inducing a bipartite structure on the set of connected components of $\reg(V)$ (note that the unit normal to $\reg(V)$ is also extended across immersed classical singularities). Furthermore, if any regular part where to have multiplicity $p/2$, then there can be no classical singularities, which forces $V$ to be smoothly embedded outside a codimension $\geq 7$ singular set, and so it is orientable.
	 
	\item [(d)] There exist $V\in \S_Q$ which are not alternating. Indeed, examples of branched stable minimal hypersurfaces which are modelled on $\text{Re}(z^{3/2})$ have been constructed (see \cite{SW07}) which cannot be alternating.
\end{enumerate}

\vspace{-0.5em}
\begin{defn}
Write $\S_Q^\pm$ for the set of all $V\in \S_Q$ which are alternating and for which all density $Q$ classical singularities are immersed.
\end{defn}
\vspace{-0.5em}

By Lemma \ref{lemma:alternating-check}, the natural varifold associated to an area minimising hypersurface mod $p$ in $B^{n+1}_2(0)$ with all classical singularities immersed belongs to $\S_{p/2}^\pm$.

\subsection{Main Results for $\S_Q^\pm$}

Our main $\eps$-regularity theorem is for the class $\S_Q^\pm$. Recall the notation for the \emph{$L^2$-height excess} $\hat{E}_V:= \Bigl(\int_{\R\times B^n_1(0)}|x^1|^2\, \ext\|V\|(x)\Bigr)^{1/2}$.

\begin{theorem}\label{thm:eps-reg}
	There exists $\eps = \eps(n,Q)\in (0,1)$ such that the following is true. Suppose $V\in \S_Q^{\pm}$ is such that $(2^n\w_n)^{-1}\|V\|(B^{n+1}_2(0))<Q+\frac{1}{2}$ and
	$$\dist_\H(\spt\|V\|\cap (\R\times B^n_1(0)), \{0\}\times B^n_1(0))<\eps.$$
	Then, for some $q\in \{1,\dotsc,Q\}$, there exist smooth functions $u_1,\dotsc,u_q:B^n_{1/2}(0)\to \R$ solving the minimal surface equation for which
	$$V\res (\R\times B^n_{1/2}(0)) = \sum^q_{i=1}\mathbf{v}(u_i) \qquad \text{and} \qquad \|u_i\|_{C^\ell} \leq C(n,Q,\ell)\hat{E}_V \quad \text{for all }\ell\geq 0.$$
\end{theorem}

The assumptions of Theorem \ref{thm:eps-reg} imply that $V$ must be close, as varifolds, to a plane of some integer multiplicity $q\leq Q$ in $\R\times B^{n}_1(0)$. If $q<Q$, then the conclusion of Theorem \ref{thm:eps-reg} is already known -- see \cite[Theorem B]{MW24} (the special case for area minimisers mod $p$ follows from \cite{White84}). Therefore, we only need to focus on the case when $q=Q$; in particular, $Q$ must be an integer.

Clearly Theorem \ref{thm:eps-reg} implies Theorem \ref{thm:main-no-bp}. From Theorem \ref{thm:eps-reg} we can also prove a regularity and compactness theorem, namely:

\begin{theorem}\label{thm:main-reg-2}
	Suppose $(V_j)_j\subset \S_Q^\pm$ satisfy $\limsup_j\|V_j\|(B^{n+1}_2(0))<\infty$. Then, there exists a subsequence $\{j^\prime\}$ and $V\in \S_Q^\pm$ for which $V_{j^\prime}\rightharpoonup V$ as varifolds in $B^{n+1}_2(0)$ and:
	\begin{enumerate}
		\item [\textnormal{(i)}] $\spt\|V\| = \mathcal{R}\cup W_{1}\cup W_{2}$, where:
		\begin{itemize}
			\item $\mathcal{R}$ consists of all regular points of multiplicity $\leq Q$ and smoothly immersed points of multiplicity $Q$ (in particular $\mathcal{R}$ is open);
			\item $W_{1}$ consists of all singular points where every tangent cone $\BC$ has spine dimension $\leq n-3$, the multiplicity of each connected component of $\reg(\BC)$ is $\leq Q$, all classical singularities in $\BC$ (if there are any) have density $Q$, and $\BC$ is not supported on a union of hyperplanes;
			\item $W_{2}$ consists of all singular points where at least one tangent cone $\BC$ either is a hyperplane of multiplicity $>Q$, or contains a classical singularity with density $>Q$, or is a cone with spine dimension $\leq n-3$ and has a regular part has multiplicity $>Q$. 
		\end{itemize}
		\item [\textnormal{(ii)}] $V\res \mathcal{R}$ is represented by a two-sided smooth immersion, and $V_{j^\prime}\to V$ locally smoothly (as an immersion) on $\mathcal{R}$.
	\end{enumerate}
\end{theorem}

We stress that since $\S_Q^\pm$ does not restrict the nature of classical singularities of density $>Q$, one does not have control over the limit $V$ in the region $\{\Theta_V\geq Q+1/2\}$. Notice that $W_2\subset\{\Theta_V\geq Q+1/2\}$, and so $V\res \{\Theta_V<Q+1/2\}$ is smoothly immersed away from a relatively closed subset of $\{\Theta_V<Q+1/2\}$ with Hausdorff dimension $\leq n-3$. Further conclusions do not appear to be possible without further restricting the types of classical singularities and branch points with density $>Q$ or strengthening $(\S2)$ to assuming stability as an immersion (cf.~\cite{MX26}).

Note that since area minimisers mod $p$ form a compact class (see \cite[Proposition 5.2]{DLHMS20}), Theorem \ref{thm:main-reg-2} contains as a special case Theorem \ref{thm:main-reg} and Theorem \ref{thm:main-compactness}. Indeed, for $T$ an area minimiser mod $p$ the corresponding set $W_2$ in Theorem \ref{thm:main-reg-2} is always empty, and since $\dim_\H(W_1)\leq n-3$ we see that $T$ is always smoothly immersed outside a relatively closed set of dimension $\leq n-3$. Hence, we are left with proving Theorem \ref{thm:eps-reg} and Theorem \ref{thm:main-reg-2}.

\section{Proof of Theorem \ref{thm:eps-reg} ($\eps$-Regularity)}\label{sec:proof}

The starting point for our proof of Theorem \ref{thm:eps-reg} is the $\eps$-regularity theorem for $\S_Q$ from \cite[Theorem A]{MW24}. This gives that, under the set-up of Theorem \ref{thm:eps-reg}, when $V$ is close to $\{0\}\times\R^n$ with multiplicity $Q$, one can represent $V\res (\R\times B_{1/2}^n(0))$ as the graph of a generalised-$C^{1,\alpha}$ \emph{multi-valued} function $u:B^n_{1/2}(0)\to \A_Q(\R)$. When all density $Q$ classical singularities are immersed, this function is $C^{1,\alpha}$ in the classical sense of multi-valued functions.

Thus, our proof is concerned with showing that when $V\in\S_Q^\pm\subset\S_Q$ this $C^{1,\alpha}$ multi-valued function globally decomposes as $Q$ single-valued minimal graphs.

\subsection{$C^{1,\alpha}$ Selections} Our first lemma provides the (global) $C^{1,\alpha}$-selection of the multi-valued function $u$. This is the only place we utilise the alternating condition in the proof of Theorem \ref{thm:eps-reg}.

\begin{lemma}\label{lemma:decomposition}
	Suppose $V\in \S_Q^\pm$ is such that there is a generalised-$C^{1,\alpha}$ function $u:B^n_{1/2}(0)\to \A_Q(\R)$ with $V\res (\R\times B^n_{1/2}(0)) = \mathbf{v}(u)$. Then, there exist $C^{1,\alpha}$ functions $u_1,\dotsc,u_Q: B^n_{1/2}(0)\to \R$ for which $$u = \sum^Q_{i=1}\llbracket u_i\rrbracket.$$ Furthermore, each $u_i$ solves the minimal surface equation away from the (relatively closed) set
	$$\mathcal{K} := \{x\in B_{1/2}^n(0): u_i(x) = u_j(x) \text{ and }Du_i(x) = Du_j(x) \text{ for all }i,j\}.$$
	If $\mathcal{K} = B^n_{1/2}(0)$, then $u_i\equiv u_j$ for all $i,j$ and (all) $u_i$ solve the minimal surface equation on $B^n_{1/2}(0)$.
\end{lemma}

\begin{proof}
	By definition of $\S^\pm_Q$, there are no classical singularities of density $<Q$ and all classical singularities of density $Q$ are immersed. It therefore follows that $u$ is $C^{1,\alpha}$. Since $u$ is $\A_Q(\R)$-valued, taking the natural ordering induced by $\R$ on the values of $u$ produces a Lipschitz selection, and so we may write $u = \sum^Q_{i=1}\llbracket \widetilde{u}_i\rrbracket$, where $\widetilde{u}_1\leq \cdots \leq \widetilde{u}_Q$ on $B_{1/2}^n(0)$.
	
	If $u \equiv Q\llbracket \widetilde{u}\rrbracket$ for some (necessarily $C^{1,\alpha}$) $\widetilde{u}:B^n_{1/2}(0)\to \R$, then as $V$ is stationary we must have that $\widetilde{u}$ is minimal on $B^n_{1/2}(0)$ and we are done. So, we may assume that $\mathcal{K}\neq B^n_{1/2}(0)$; in particular, $\widetilde{u}_i\not\equiv \widetilde{u}_j$ for some $i,j$.
	
	From \cite[Theorem A]{MW24}, we may write $B^n_{1/2}(0) = \mathcal{R}\cup \mathcal{C}\cup\mathcal{K}$ as a disjoint union, where:
	\begin{itemize}
		\item $\mathcal{R} = \{x\in B^n_{1/2}(0) : u(x)\neq Q\llbracket a\rrbracket\text{ for any }a\in \R\}$ (in particular, from the properties of $V\in \S_Q$, $u$ has a smooth selection locally about any point in $\mathcal{R}$);
		\item $\mathcal{C}$ the set of (density $Q$) classical singularities of $u$;
		\item $\mathcal{K}$ is the (density $Q$) critical set of $u$ (once the lemma is proven, $\mathcal{K}$ is as in the statement).
	\end{itemize}
	Let $\nu$ be (a choice of) the unit normal to $\reg(V)\cap\{\Theta_V<Q\}$ given by the alternating condition. In particular, if $\nu_i(x)$ denotes the corresponding unit normal to $\graph(\widetilde{u}_i)$ at $x\in \mathcal{R}$, then $\text{sign}(\nu_i(x)\cdot e_{1})$ is independent of $i$ (note that $e_1$ is the vertical direction); we write $s(x):\mathcal{R}\to \{\pm 1\}$ for this sign. The alternating condition gives two properties:
	\begin{enumerate}
		\item [(i)] $s$ is constant on connected components of $\mathcal{R}$;
		\item [(ii)] If $x\in\mathcal{C}$ and $R_1,R_2$ are the two distinct connected components of $\mathcal{R}$ with $x\in \overline{R}_1\cap \overline{R}_2$, then $s(R_1)s(R_2) = -1$, where $s(R_i)$ is the value of $s$ on $R_i$.
	\end{enumerate}
	Now, for $j=1,\dotsc,Q$, define functions $w_j:B^n_{1/2}(0)\to \R$ by:
	\begin{equation}\label{eq:defn of C^1,al selection}
	    w_j(x) := \begin{cases}
		\widetilde{u}_j(x) & \text{if }x\in\mathcal{R}\text{ and }s(x) = +1;\\
		\widetilde{u}_{Q-j+1}(x) & \text{if }x\in\mathcal{R}\text{ and }s(x) = -1;\\
		\widetilde{u}_1(x) & \text{if }x\in \mathcal{C}\cup\mathcal{K}.
	\end{cases}
	\end{equation}
	Notice that as $\widetilde{u}_i \equiv \widetilde{u}_j$ on $\mathcal{C}\cup\mathcal{K}$ for all $i,j$, the choice of $\widetilde{u}_i$ on $\mathcal{C}\cup\mathcal{K}$ is irrelevant. 
	
	Clearly $w_j:B^n_{1/2}(0)\to \R$ is well-defined and Lipschitz. We claim that in fact it is $C^{1,\alpha}$. Indeed, for this one only needs to check that $w_j$ is $C^1$, as then the $C^{1,\alpha}$ regularity of $u$ implies that $w_j$ is $C^{1,\alpha}$ (cf.~the proof of Lemma \ref{lemma:regularity}). By definition of $u$ being $C^{1,\alpha}$ we know $w_j$ is $C^1$ at every point of $\mathcal{R}\cup\mathcal{K}$, and so we just need to check $w_j$ is $C^1$ at $x\in \mathcal{C}$. 
	
	Fix $x\in \mathcal{C}$. There is then a radius $\rho = \rho(x)>0$ such that we can partition $B_\rho(x) = \Omega_+\cup \Omega_-\cup \Gamma$, where $\Gamma\subset B_\rho(x)$ is an $(n-1)$-dimensional submanifold of $B_\rho(x)$ with $\del\Gamma\subset\del B_\rho(x)$. By relabelling we can without loss of generality assume that $s=\pm 1$ on $\Omega_\pm$, respectively. Write $\widetilde{u}_i^\pm := \widetilde{u}_i|_{\Omega_\pm}\in C^{1,\al}(\overline{\Omega_\pm})$. Then from the Hopf boundary point lemma and the definition of a classical singularity (as the only intersections occur along the image of $\Gamma$) it follows that, if $\nu_\Gamma$ is the unit normal to $\Gamma$ directed into $\Omega_+$, then
	$$D_{\nu_\Gamma} \widetilde{u}_1^+ \leq D_{\nu_\Gamma} \widetilde{u}_2^+ \leq\cdots \leq D_{\nu_\Gamma} \widetilde{u}_Q^+ \quad \text{on }\Omega_+ \qquad \text{and} \qquad D_{\nu_\Gamma} \widetilde{u}_1^- \geq D_{\nu_\Gamma} \widetilde{u}_2^- \geq \cdots \geq D_{\nu_\Gamma} \widetilde{u}_Q^- \quad \text{on }\Omega_-.$$
	Since the tangent cones to $V$ at points along $\Gamma$ are sums of $Q$ hyperplanes, this therefore implies that $D_\nu \tilde{u}_{Q-j+1}^+ = D_\nu \tilde{u}_j^-$ for each $j$. Furthermore, all the derivatives of the $u^\pm_i$ parallel to $\Gamma$ agree by definition of a classical singularity. Hence, we have $D\tilde{u}_{Q-j+1}^+ = D\tilde{u}_j^-$ along $\Gamma$ for all $j=1,\dotsc,Q$, which implies that $w_j$ is $C^1$ at $x$. Thus, $w_j$ is $C^1$ on all of $B^n_{1/2}(0)$.
	
	We now know that each $w_j$ is $C^{1,\alpha}$ on $B^n_{1/2}(0)$. But clearly $u = \sum^Q_{j=1}\llbracket w_j\rrbracket$, and so this proves the first claim of the lemma.
	
	For the second claim, we want to show that the $w_j$ solve the MSE away from $\mathcal{K}$, i.e.~on $\mathcal{R}\cup \mathcal{C}$.
	The fact that $w_j$ solves the MSE on $\mathcal{R}$ is clear, as they are regular points of $V$. Thus to conclude it suffices to show that $w_j$ solves the MSE locally about any point $y\in \mathcal{C}$. Since $\mathcal{C}$ is a $(n-1)$-dimensional $C^1$ submanifold and $w_j$ is $C^1$ this is a standard extendability fact, although we include the details for completeness.
	
	For simplicity, let us write $w = w_j$. We may translate and rescale to assume without loss of generality that $y = 0$ and $w$ is a $C^1$ function on $B^n_1(0)$ which is smooth and solves the MSE pointwise on $B^n_1(0)\setminus \Gamma$, where $\Gamma\subset B^n_1(0)$ is an $(n-1)$-dimensional $C^1$ submanifold which disconnects $B^n_1(0)$ into two connected components $\Omega_\pm$. Write $F(p) := p/(1+|p|^2)^{1/2}$, so that the MSE can be written as $\div(F(Dw)) = 0$, which we know to be satisfied pointwise on $B^n_1(0)\setminus \Gamma$.
	
	Fix $\eps>0$ small and let $\phi\in C^1_c(B_1^n(0))$. Write $N_\eps(\Gamma) = \{x\in B^n_1(0):\dist(x,\Gamma)<\eps\}$ for the $\eps$-neighbourhood of $\Gamma$ and set $\Omega_\eps:= B^n_1(0)\setminus N_\eps(\Gamma)$. By the divergence theorem we have:
	\begin{align*}
	\int_{\Omega_\eps}F(Dw)\cdot D\phi & = -\int_{\Omega_\eps}\phi\,\div(F(Dw)) + \int_{\del\Omega_\eps}\phi F(Dw)\cdot\nu_\eps\\
	& = \int_{\Omega_+\cap \del\Omega_\eps}\phi F(Dw)\cdot\nu_\eps^+ + \int_{\Omega_-\cap \del\Omega_\eps}\phi F(Dw)\cdot\nu_\eps^-
	\end{align*}
	where $\nu_\eps$ is the outward pointing unit normal to $\del\Omega_\eps$, with $\nu_\eps^\pm$ the outward pointing unit normals to $\Omega_\pm\cap \del\Omega_\eps$, respectively. Since $F(Dw)$ is continuous and $\Gamma$ is $C^1$, if we take $\eps\to 0$ we see that
	$$\int_{\Omega_\pm \cap\del\Omega_\eps}\phi F(Dw)\cdot\nu_\eps^{\pm} \to \pm \int_\Gamma \phi F(Dw)\cdot\nu_\Gamma$$
	where $\nu_\Gamma$ is the unit normal to $\Gamma$ directed into $\Omega_+$. Hence, these two terms cancel in the limit as $\eps\to 0$, and thus we see that
	$$\int F(Dw)\cdot D\phi = 0 \qquad \text{for all }\phi\in C^1_c(B^n_1(0)).$$
	Hence, $w$ is a $C^1$ weak solution to the MSE in $B_1^n(0)$, and thus standard elliptic regularity gives that $w$ is smooth and thus pointwise solves the MSE in $B^n_1(0)$. Hence each $w_i$ is smooth and pointwise solves the MSE away from $\mathcal{K}$, completing the proof.
\end{proof}

What remains to be shown for Theorem \ref{thm:eps-reg} is that each $u_i$ in Lemma \ref{lemma:decomposition} is minimal \emph{across} $\mathcal{K}$. We stress that this is not a priori clear, as the size of $\mathcal{K}$ is unknown at present (cf.~Remark \ref{remark:decomp}).

Next, we prove that any coarse blow-up generated by a sequence $(V_j)_j\subset \S_Q^{\pm}$ converging to a multiplicity $Q$ hyperplane can be written as a sum of $Q$ (smooth, single-valued) harmonic functions. We write $\mathfrak{B}_Q^\pm$ for the set of all coarse blow-ups generated by blowing-up sequences in $\S_Q^{\pm}$ converging to the hyperplane $\{0\}\times\R^n$ with multiplicity $Q$.

\begin{lemma}\label{lemma:blow-ups}
	Let $v\in \mathfrak{B}^\pm_Q$. Then there exist (smooth) harmonic functions $v_1,\dotsc,v_Q: B^n_1(0)\to \R$ with $v = \sum^Q_{i=1}\llbracket v_i\rrbracket$.
\end{lemma}

\begin{proof}
Let $(V_j)_j\subset \mathfrak{B}_Q^\pm$ be the sequence generating $v\in \mathfrak{B}_Q^\pm$. In particular, $V_j$ converges to the hyperplane $\{0\}\times\R^n$ with multiplicity $Q$ in $\R\times B^n_1(0)$. Fixing a sequence $\sigma_j\uparrow 1$ and passing to an appropriate subsequence of $(V_j)_j$, we may apply \cite[Theorem A]{MW24} to represent $V_j$ on $\R\times B_{\sigma_j}^n(0)$ by the graph of a generalised-$C^{1,\alpha}$ function $u_j: B_{\sigma_j}^n(0)\to \A_Q(\R)$.
	
By Lemma \ref{lemma:decomposition}, we can write $u_j = \sum^Q_{i=1}\llbracket u_j^{(i)}\rrbracket$, where $u_j^{(i)}:B^n_1(0)\to \R$ is $C^{1,\alpha}$. Moreover, for any $\sigma\in (0,1)$, if $j$ is sufficiently large we know from \cite[Theorem A]{MW24} that\footnote{This estimate follows from the estimate in \cite[Theorem A]{MW24} coupled with the fact that the $u^{(i)}_j$ solve the MSE away from the $Q$-touching set of $u_j$ by Lemma \ref{lemma:decomposition}. Indeed, one can argue similarly to the proof of Lemma \ref{lemma:regularity} to see this, although we only give the details once there (as it is the harder case). We stress that this estimate on the constituents $u^{(i)}_j$ does not trivially follow from the corresponding estimate for the $Q$-valued function $u_j$ as the $C^{0,\alpha}$ semi-norm for multi-valued functions involves the Almgren metric $\mathcal{G}$.}
	$\|u_j^{(i)}\|_{C^{1,\alpha}(B_{\sigma}^n(0))} \leq C(n,Q,\sigma)\hat{E}_{V_j}$. Thus, by Arzelà--Ascoli, we may pass to a subsequence for which $\hat{E}_{V_j}^{-1}u_j^{(i)}\to v^{(i)}$ in $C^{1}_{\text{loc}}(B^n_1(0))$ for each $i=1,\dotsc,Q$, for some $v^{(i)}: B^n_1(0)\to \R$ which are $C^{1,\alpha}_{\text{loc}}(B^n_1(0))$. We know that $v$ is the $L^2$ limit of $v_j:= \hat{E}_{V_j}^{-1}u_j$, and thus $v = \sum^Q_{i=1}\llbracket v^{(i)}\rrbracket$.
	
If we write $\mathcal{K}_j:=\big\{x\in B_{\sigma_j}^n(0):u_j^{(k)}(x) = u_j^{(\ell)}(x)\text{ and }Du_j^{(k)}(x) = Du_j^{(\ell)}(x)\text{ for all }k,\ell\big\}$, then from the $C^1_{\text{loc}}$ convergence of $u_j^{(i)}$ to $v^{(i)}$ we see that
$$\mathcal{K}_j\to \mathcal{K}_* \qquad \text{locally in Hausdorff distance in }B^n_1(0)$$
for some relatively closed subset $\mathcal{K}_*\subset B^n_1(0)$ obeying $\mathcal{K}_*\subseteq\mathcal{K} := \big\{x\in B^n_1(0):v^{(k)}(x) = v^{(\ell)}(x)\text{ and }Dv^{(k)}(x) = Dv^{(\ell)}(x) \text{ for all }k,\ell\big\}$. Furthermore, as the $u_j^{(i)}$ solve the MSE away from $\mathcal{K}_j$ by Lemma \ref{lemma:decomposition} (thus giving improved regularity and estimates), we see that $v^{(i)}$ must be harmonic on $B^n_1(0)\setminus\mathcal{K}$. Of course, if $\mathcal{K} = B^n_1(0)$ then $v^{(i)}\equiv v^{(\ell)}$ for all $i,\ell$, and thus $v = Q\llbracket v^{(1)}\rrbracket$. But since the average of a coarse blow-up of stationary integral varifolds is always harmonic (see \cite[$(\mathfrak{B}3)$]{MW24}), this would imply that $v^{(1)}$ is harmonic, completing the proof. So, we may assume that $\mathcal{K}\neq B^n_1(0)$.
	
We claim that $\dim_\H(\mathcal{K})\leq n-2$. Once we have this, since each $v^{(i)}$ is $C^{1,\alpha}$ and satisfies $\Delta v^{(i)}=0$ on $B^n_1(0)\setminus \mathcal{K}$, it follows that $\mathcal{K}$ is removable from this equation and thus $v^{(i)}$ is harmonic on all of $B^n_1(0)$, completing the proof.
	
To see that $\dim_\H(\mathcal{K})\leq n-2$, since $\mathcal{K}\neq B^n_1(0)$ there exist $k,\ell$ for which $v^{(k)}\not\equiv v^{(\ell)}$; set $w:= v^{(k)} - v^{(\ell)}$. Then $w$ is a non-zero $C^{1,\alpha}$ function on $B^n_1(0)$ which solves $\Delta w = 0$ on $B^n_1(0)\setminus \mathcal{K}$, where $\mathcal{K} \subseteq \{w=0, Dw=0\}$. But as the zero set is removable from the equation $\Delta w =0$ for a $C^{1,\alpha}$ function (see e.g.~\cite{Kral83}), it follows that we must have $\Delta w = 0$ on $B^n_1(0)$, giving that $w$ is harmonic (and thus smooth) on all of $B^n_1(0)$. Standard estimates from (quantitative) unique continuation now give that $\dim_\H(\{w=0,Dw=0\})\leq n-2$, and so as $\mathcal{K}\subseteq \{w=0,Dw=0\}$ necessarily we have $\dim_\H(\mathcal{K})\leq n-2$ as well. This completes the proof.
\end{proof}

\begin{remark}
At the end of the proof of Lemma \ref{lemma:blow-ups}, we used that if $w$ is a $C^{1,\alpha}$ function on $B^n_1(0)$ which solves $\Delta w = 0$ on $B^n_1(0)\setminus \{w=0,Dw=0\}$, then in fact $\Delta w = 0$ on $B^n_1(0)$. This was then used to deduce that $\dim_\H(\CK)\leq n-2$. In the present setting, one can argue for this dimension bound directly, avoiding the use of an extension result like \cite{Kral83} (which in fact only needs $w$ to be $C^1$). Indeed, one can show that $w$ has a monotone frequency function and thus use standard methods to directly prove $\dim_\H(\{w=0,Dw=0\})\leq n-2$ (at which point the proof of Lemma \ref{lemma:blow-ups} is complete). The fact that $w$ has a monotone frequency function can be seen by extending the two key variational identities (squash and squeeze) across the set $\{w=0,Dw=0\}$ by using a cut-off argument and that $w$ must be in $W^{2,2}_{\text{loc}}(B^n_1(0))$ (analogously to that seen in \cite[Section 2]{SW16}). This could be viewed as a much simpler instance of the result in \cite[Appendix A]{MW24}. 
\end{remark}

\subsection{$C^{1,1}$ Regularity via Planar Frequency}

Throughout this subsection we suppose $V\in \S_Q^{\pm}$ is such that there are $C^{1,\alpha}$ functions $u_1,\dotsc,u_Q:B^n_{1/2}(0)\to \R$ for which, writing $u = \sum^Q_{i=1}\llbracket u_i\rrbracket$,
$$V \res (\R\times B^n_{1/2}(0)) = \mathbf{v}(u) \equiv \sum^Q_{i=1}\mathbf{v}(u_i).$$
We also assume that $\hat{E}_V<\eps$ for some suitably small $\eps = \eps(n,Q)\in (0,1)$. By \cite[Theorem A]{MW24}, this ensures $\sup|Du| < C(n,Q)\eps$ is small.

We will show that each $u_i$ must in fact be $C^{1,1}$. For this, we utilise several properties of the planar frequency function established in \cite{KMW}. As usual, we write $\mathcal{K}$ for the $Q$-touching set of $u$, i.e.
$$\mathcal{K} := \{x\in B^n_{1/2}(0): u_i(x) = u_j(x) \text{ and }Du_i(x) = Du_j(x) \text{ for all }i, j\}.$$

\begin{lemma}\label{lemma:lower-bound}
	Any tangent map to $u$ at a point in $\mathcal{K}$ is collection of $Q$ homogeneous harmonic functions with homogeneity $\geq 2$.
\end{lemma}

\begin{proof}
	We know that tangent maps are homogeneous of degree $>1$ by \cite[Corollary 3.4 \& Corollary 3.6]{KMW}. As tangent maps are a special case of a coarse blow-up, by Lemma \ref{lemma:blow-ups} we also know that the tangent map is comprised of $Q$ harmonic functions. These two facts prove the result, as harmonic functions can only have integer homogeneities.
\end{proof}

The following proposition establishes a refined structure of tangent maps in the present setting (an analogous result in the minimising mod $p$ setting appears in \cite[Theorem 3.7]{DLHMS+22}). We will not need it for the proof of Theorem \ref{thm:eps-reg}, but will discuss it further in Section \ref{sec:rest}.

\begin{prop}\label{prop:relations}
	Let $v$ be a tangent map to $u$ at a point in $\mathcal{K}$. Then, $v = \sum^Q_{i=1}\llbracket p+c_iq\rrbracket$, where $p,q$ are $\beta$-homogeneous harmonic functions for some $\beta\geq 2$ and $c_i\in \R$ are constants.
\end{prop}

\begin{proof}
	From Lemma \ref{lemma:lower-bound} we know $v = \sum^Q_{i=1}\llbracket v_i\rrbracket$ where each $v_i$ is a (possibly zero) $\beta$-homogeneous harmonic function for some $\beta\geq 2$.
	
	Set $p := v_1$. If $v = Q\llbracket p\rrbracket$, then set $q = 0$ and the proposition is established. Thus, we may assume that there is some $j\neq 1$ with $v_j\not\equiv v_1$. Set $q:= v_j-v_1$. We claim that the proposition holds with this choice of $p,q$.
	
	For each $i\neq 1$ set $w_i:= v_i-v_1$, and identify $w_i$ with its homogeneous extension to $\R^n$. If $w_i\equiv 0$ then set $c_i = 0$. Otherwise, $w_i,w_j$ are both non-zero (note that $j$ has been fixed here). Since from \cite[Theorem 3.12]{MW24} we know that the only classical singularities and touching points in $v$ are density $Q$, this implies that the zero sets of $w_i,w_j$ must coincide. Consider a connected component $\Omega$ of $\{w_i \neq 0\} = \{w_j\neq 0\}$ (such a component must exist as $w_i\not\equiv 0$). Set $\Sigma:= \Omega\cap S^{n-1}$. Then $w_i|_\Sigma$ and $w_j|_\Sigma$ are Dirichlet eigenfunctions of $\Delta_{S^{n-1}}$ on $\Sigma$ which have a sign. As the space of such eigenfunctions is $1$-dimensional (being multiples of the first eigenfunction) we therefore see that on $\Omega$ we must have $w_i = c_iw_j$ for some constant $c_i$. But as $\Omega$ is open, standard unique continuation for harmonic functions implies that we must have $w_i = c_iw_j$ on all of $\R^n$. Hence, $v_i - v_1 = c_i(v_j-v_1)$, i.e.~$v_i = p + c_iq$. This completes the proof.
\end{proof}

\begin{lemma}\label{lemma:regularity}
	$u_i\in C_{\textnormal{loc}}^{1,1}(B^n_{1/2}(0))$ and $\|u_{i}\|_{C^{1,1}(B^n_{1/4}(0))} \leq C\|u\|_{L^2(B_{1/2}^n(0))}$ for all $i=1,\dotsc,Q$.
\end{lemma}

\textbf{Note:} The constant $C$ in Lemma \ref{lemma:regularity} depends on an upper bound for the planar frequency in a neighbourhood, not just on $n,Q$. This won't impact the arguments to follow.

\begin{proof}
	If $\mathcal{K} = \emptyset$, then by Lemma \ref{lemma:decomposition} we know that $u_i$ solves the MSE and the lemma trivially follows from elliptic regularity. So, we may assume that $\mathcal{K}\neq \emptyset$.
	
	Suppose first that $u(0) = Q\llbracket 0\rrbracket$ and $Du(0) = Q\llbracket 0\rrbracket$. Then by \cite[Corollary 3.4]{KMW} we have for any $\rho\in (0,7/16)$,
	$$\rho^{-n}\int_{B^n_\rho(0)}|u|^2 \leq C\rho^{2\mathcal{N}_V(0)}\|u\|_{L^2(B_{15/32}^n(0))}^2$$
	where $\mathcal{N}_V(0)$ is the planar frequency of $V$ at $0\in \spt\|V\|$, and $C = C(n,Q,\mathcal{N}_V(0))$. From \cite[Corollary 3.6]{KMW} we know that $\mathcal{N}_V(0)$ is determined by the homogeneity of the tangent map to $V$ at $0$, which is $\geq 2$ by Lemma \ref{lemma:lower-bound}. Thus,
	\begin{equation}\label{E:freq-1}
		\int_{B_\rho^n(0)}|u|^2 \leq C\rho^{n+4}\|u\|_{L^2(B^n_{15/32}(0))}^2 \qquad \text{for all }\rho\in (0,7/16).
	\end{equation}
	Using (the rescaled version of) the estimate in \cite[Theorem A]{MW24} combined with \eqref{E:freq-1} we then have, for each $\rho\in (0,13/32)$,
	\begin{equation}\label{E:freq-2}
		\sup_{B^n_\rho(0)}|Du| + \rho^{-1}\sup_{B^n_\rho(0)}|u| \leq C\left(\rho^{-n-2}\int_{B^n_{14\rho/13}(0)}|u|^2\right)^{1/2} \leq C\rho\|u\|_{L^2(B^n_{15/32}(0))}.
	\end{equation}
	Now at a general $y\in \mathcal{K}\cap B_{7/16}^n(0)$, one can apply \eqref{E:freq-2} to the multi-valued function $\tilde{u}_y$ defined over the (unique, multiplicity $Q$) affine tangent plane to $V$ at $(u_1(y),y)$ whose graph coincides with $V$. Using that $|Du|$ is small, we may rewrite the resulting expression in terms of $u$ to arrive at
	\begin{equation*}
		\sup_{B^n_\rho(y)}|Du-DL_y| + \rho^{-1}\sup_{B^n_\rho(y)}|u-L_y| \leq C\rho\left(\int_{B^n_{1/2}(0)}|u-L_y|^2\right)^{1/2} \qquad \forall\rho\in (0,\tfrac{7}{16}\left(\tfrac{1}{2}-|y|\right)).
	\end{equation*}
	where $L_y(x):= u_i(y) + (x-y)\cdot Du_i(y)$ for any $i=1,\dotsc,Q$ (recall that $u_i(y) = u_j(y)$ and $Du_i(y) = Du_j(y)$ for all $i,j$). In particular, we have for any $i=1,\dotsc,Q$ and $y\in \mathcal{K}\cap B_{7/16}^n(0)$:
	\begin{equation}\label{E:freq-4}
		\sup_{B^n_\rho(y)}|Du_i-DL_y| + \rho^{-1}\sup_{B^n_\rho(y)}|u_i - L_y| \leq C\rho\left(\int_{B^n_{1/2}(0)}|u-L_y|^2\right)^{1/2} \qquad \forall \rho\in (0,\tfrac{7}{16}\left(\tfrac{1}{2}-|y|\right)),
	\end{equation}
	We stress that, whilst the constant $C$ depends on (an upper bound for) the planar frequency of $V$ at $y$, by upper semi-continuity of planar frequency \cite[Corollary 3.5]{KMW} we can take a uniform upper bound which holds at all $y\in\mathcal{K}\cap B^n_{7/16}(0)$, and so $C$ only depends on this global upper bound.
	
	Now fix any $x,y\in B^n_{1/16}(0)$ distinct and fix $i\in \{1,\dotsc,Q\}$. Suppose without loss of generality $\dist(y,\mathcal{K})\leq \dist(x,\mathcal{K})$, and take $y_*\in\mathcal{K}$ with $|y-y_*| = \dist(y,\mathcal{K})$ (if $\dist(y,\mathcal{K}) = 0$, simply take $y= y_*$). For simplicity, we write $A:= C\|u-L_{y_*}\|_{L^2(B^n_{1/2}(0))}$. We consider two cases:
	\begin{enumerate}
		\item [(1)] $\frac{1}{4}\dist(y,\mathcal{K})\leq |x-y|$;
		\item [(2)] $\frac{1}{4}\dist(y,\mathcal{K})>|x-y|$.
	\end{enumerate}
	Consider first case (1). Then we have $|y-y_*| \leq 4|x-y|$, and \eqref{E:freq-4} gives
	\begin{align}
	\nonumber|Du_i(x)-Du_i(y)| & \leq |Du_i(x) - Du_i(y_*)| + |Du_i(y_*)-Du_i(y)|\\
	& \leq A|x-y_*| + A|y-y_*|\nonumber\\
	& \leq A|x-y| + 2A|y-y_*|\nonumber\\
	& \leq 9A|x-y|.\label{E:freq-6}
	\end{align}
	
	In the case (2), set $\rho:= |x-y|$. We know $B^n_{\dist(y,\mathcal{K})}(y)\cap \mathcal{K} = \emptyset$, and thus from Lemma \ref{lemma:decomposition} we know that $u_i$ solves the MSE on $B^n_{\dist(y,\mathcal{K})}(y)$. Also, $B^n_{4\rho}(y)\subset B^n_{\dist(y,\mathcal{K})}(y)$ by assumption. Thus, by standard estimates for the MSE (or, equivalently, one can apply the equivalent estimate to \eqref{E:freq-4} with $u_i$ in place of $u$ on the ball $B^n_{\dist(y,\mathcal{K})}(y)$) we have
	\begin{align*}
	|Du_i(x)-Du_i(y)| & \leq \rho^{-1}|x-y|\left((2\rho)^{-n-2}\int_{B^n_{2\rho}(y)}|u_i-L_{y}|^2\right)^{1/2}\\
	& \leq  C|x-y|\left(\dist(y,\mathcal{K})^{-n-4}\int_{B^n_{\dist(y,\mathcal{K})}(y)}|u_i-L_y|^2\right)^{1/2}.
	\end{align*}
	Using \eqref{E:freq-4} with $y_*$ in place of $y$ and $2\dist(y,\mathcal{K})$ in place of $\rho$ we then have
	\begin{align*}
		\int_{B^n_{\dist(y,\mathcal{K})}(y)}|u_i-L_y|^2 & \leq 2\int_{B^n_{2\dist(y,\mathcal{K})}(y_*)}|u_i-L_{y_*}|^2 + |L_y-L_{y_*}|^2\\
		& \leq C\dist(y,\mathcal{K})^{n+4}\int_{B^n_{1/2}(0)}|u-L_{y_*}|^2 + 2\int_{B^n_{2\dist(y,\mathcal{K})}(y_*)}|L_y-L_{y_*}|^2
	\end{align*}
	To control the last term here, note that
	\begin{align*}
	L_y(z) - L_{y_*}(z) = (u_i(y) - L_{y_*}(y)) + (z-y)\cdot (Du_i(y)-Du_i(y_*))
	\end{align*}
	and so for $z\in B^n_{2\dist(y,\mathcal{K})}(y_*)$, using \eqref{E:freq-4} and $|z-y| \leq 3\dist(y,\mathcal{K})$ we have the pointwise bound
	\begin{align*}
		|L_y(z) - L_{y_*}(z)| \leq C\dist(y,\mathcal{K})^2\left(\int_{B^n_{1/2}(0)}|u-L_{y_*}|^2\right)^{1/2}
	\end{align*}
	and hence combining the above inequalities we have
	$$\int_{B^n_{\dist(y,\mathcal{K})}(y)}|u_i-L_y|^2 \leq C\dist(y,\mathcal{K})^{n+4}\int_{B^n_{1/2}(0)}|u-L_{y_*}|^2$$
	and so
	\begin{equation}\label{E:freq-7}
		|Du_i(x) - Du_i(y)| \leq C|x-y|\left(\int_{B^n_{1/2}(0)}|u-L_{y_*}|^2\right)^{1/2}.
	\end{equation}
	Combining \eqref{E:freq-6}, \eqref{E:freq-7}, we see that, for any $x,y\in B^n_{1/16}(0)$, there is some $y_*\in B_{1/4}^n(0)$ (depending on $x,y$) for which
	$$|Du_i(x)-Du_i(y)| \leq C|x-y|\left(\int_{B^n_{1/2}(0)}|u-L_{y_*}|^2\right)^{1/2}.$$
	Finally, since $\sup_{B^n_{1/4}(0)}(|u| + |Du|) \leq C\|u\|_{L^2(B^n_{1/2}(0))}$, we have $\sup_{B^n_{1/2}(0)}|L_{y_*}| \leq C\|u\|_{L^2(B^n_{1/2}(0))}$, and hence
	$$|Du_i(x) - Du_i(y)| \leq C|x-y|\left(\int_{B^n_{1/2}(0)}|u|^2\right)^{1/2}.$$
	This shows that $Du_i\in C^{0,1}(B^n_{1/16}(0))$, with the desired estimate. By translating, one similar gets that $Du_i\in C^{0,1}_{\text{loc}}(B^n_{1/2}(0))$, with the claimed estimate on $B^n_{1/4}(0)$. This completes the proof.
\end{proof}

\subsection{Proof of Theorem \ref{thm:eps-reg}}

By \cite[Theorem A]{MW24} we know that there is a generalised-$C^{1,\alpha}$ function $u:B^n_{1/2}(0)\to \A_Q(\R)$ for which $V \res (\R\times B^n_{1/2}(0)) = \mathbf{v}(u)$. Hence by Lemma \ref{lemma:decomposition} and Lemma \ref{lemma:regularity}, we know that $u = \sum^Q_{i=1}\mathbf{v}(u_i)$ where $u_i:B^n_{1/2}(0)\to \R$ are $C^{1,1}$ and solve the MSE outside of the $Q$-touching set
$$\mathcal{K}:= \{x\in B^n_{1/2}(0):u_i(x) = u_j(x) \text{ and }Du_i(x) = Du_j(x) \text{ for all }i\neq j\}.$$
Our aim is to show that in fact the $u_i$ solve the MSE across $\mathcal{K}$ as well; the claimed estimates then follow from standard elliptic regularity. If $\mathcal{K} = B^n_{1/2}(0)$ then $u_i\equiv u_j$ for all $i, j$ and so $u = Q\llbracket u_1\rrbracket$. Hence as $V$ is stationary this implies that $u_1$ is minimal. Since $u_1$ is $C^{1,1}$ this then implies that $u_1$ is smooth and we are done.

Therefore, we may assume that $\mathcal{K}\neq B^n_{1/2}(0)$. In particular, there are $i, j$ for which $u_i\not\equiv u_j$; without loss of generality suppose $u_1\not\equiv u_2$. As $\mathcal{K}$ is not a level-set of any $u_i$ we will show that $\dim_\H(\mathcal{K})\leq n-2$; this then completes the proof as each $u_i$ is $C^{1,1}$ and (closed) sets which are $\H^{n-1}$-null are removable from the MSE.

To see that $\dim_\H(\mathcal{K})\leq n-2$, consider $f:= u_1-u_2\not\equiv 0$. Notice that $f\in C^{1,1}$ and that
$$\mathcal{K}\subseteq \{x\in B^n_{1/2}(0): f(x) = 0\text{ and }Df(x) = 0\}$$
(in fact, this is an equality in the present situation). Moreover, on $B^n_{1/2}(0)\setminus \mathcal{K}$, $f$ solves (pointwise)
$$\div(A(x)Df) = 0$$
where $A\in C^{0,1}$; more precisely, writing $F^{ij}(p):= (1+|p|^2)^{-1/2}\big(\delta_{ij} - \frac{p_ip_j}{1+|p|^2}\big)$, we have
$$A(x) = \int^1_0 F^{ij}(tDf(x) + Du_2(x))\, \ext t$$
and hence one sees that the regularity of the $u_i$ is directly tied to the regularity of $A$ (and so in order to get that $A$ is Lipschitz we needed to know that the $u_i$ are $C^{1,1}$). Now, notice that almost everywhere we have that $A$ and $Df$ are differentiable, and so almost everywhere we have
$$\div(ADf) = AD^2f + DA\cdot Df.$$
We know that this expression vanishes on $B^n_{1/2}(0)\setminus \mathcal{K}$. Also, by definition, we know that $Df=0$ on $\mathcal{K}$, and furthermore, as $\mathcal{K}$ is contained within a level-set of $Df$, we know that $\H^n$-a.e.~on $\mathcal{K}$ we must have $D^2f = 0$ (this follows by standard upper-density arguments). Thus, $\H^n$-a.e.~on $\mathcal{K}$ we have $AD^2f + DA\cdot Df = 0$, which means that $\H^n$-a.e.~on $B^n_{1/2}(0)$ we have $\div(ADf) = 0$. In particular, as $f \in C^{1,1}$ and $A\in C^{0,1}$ we know that $A Df\in C^{0,1}$ and hence $A Df\in W^{1,\infty}$, meaning that for any $\phi\in C^{\infty}_c(B^n_{1/2}(0))$ we have
$$0 = \int \div(ADf)\phi = - \int ADf\cdot D\phi$$
and so $f$ is a weak solution to $\div(ADf) = 0$. In particular, as $A\in C^{0,1}$ and $f\in W^{1,2}$, by the work of Garofalo--Lin \cite{GarofaloLin86} we know that $f$ has a monotone frequency function. From this and the $C^{1,1}$ regularity of $f$, standard arguments now give that $\dim_\H(\{f=0,Df=0\})\leq n-2$. Hence, $\dim_\H(\mathcal{K})\leq n-2$, thereby completing the proof.
\qed

\section{Proof of Regularity \& Compactness}\label{sec:other}

We first classify alternating cones which are sums of hyperplanes. Here, \emph{alternating} is analogous to Definition \ref{defn:orientation}, namely that the regular part of the cone is equipped with an orientation such that at each classical singularity of the cone, the half-hyperplanes forming the classical singularity induced the same orientation on their common boundary.

\begin{lemma}\label{lemma:spines}
	Consider $\BC = \sum^N_{i=1}q_i|P_i|$, where $P_i$ are distinct hyperplanes in $\R^{n+1}$ passing through $0$ and $N,q_i\in \Z_{\geq 1}$. Suppose that $\BC$ is alternating. Then either:
	\begin{enumerate}
		\item [\textnormal{(i)}] $N=1$, i.e.~$\BC = q_1|P_1|$ is a single hyperplane with multiplicity $q_1\geq 1$; or
		\item [\textnormal{(ii)}] $\dim(P_1\cap P_2\cap\cdots\cap P_N) = n-1$ (or equivalently $\dim(S(\BC))=n-1$).
	\end{enumerate}
\end{lemma}
\begin{proof}
	We look case-by-case on $\sfrak:=\dim(S(\BC))$, the dimension of the spine $S(\BC) \equiv P_1\cap P_2\cap \cdots \cap P_N$ of $\BC$. If $\sfrak=n$, then we are in case (i). If $\sfrak=n-1$, then we are in case (ii).
	
	Consider the case $\sfrak=n-2$. Then after an ambient rotation we may write $\BC = \BC_0\times \R^{n-2}$, where $\BC_0 = \sum^N_{i=1}q_i|\widetilde{P}_i|$ for some distinct $2$-planes $\widetilde{P}_i\subset\R^3$ passing through $0$. Consider the multiplicity one link $\Sigma:= S^2\cap (\cup_{i=1}^N \widetilde{P}_i)$, which is a geodesic network on $S^2$. 
	
	Consider the combinatorial graph associated to $\Sigma$, namely $G := (V,E)$, where:
	\begin{itemize}
	\item the vertex set $V$ is the set of all points on $S^2$ where two or more $P_i$ intersect;
	\item the edge set $E$ consists of the connected components of $\Sigma\setminus V$, namely two vertices are joined by an edge if they are the boundary points of a connected component of $\Sigma\setminus V$.
	\end{itemize}
	Note the following properties of $G$:
	\begin{enumerate}
		\item [(a)] Every vertex in $V$ has degree at least $4$ (as vertices arise from the intersection of two or more planes);
		\item [(b)] $G$ contains no odd cycles (as $\BC$ is alternating, each vertex must act as a source or a sink);
		\item [(c)] $G$ contains no cycles of order $2$.
	\end{enumerate}
	Property (c) is seen as follows. If there were a cycle of order $2$, then it would be formed by two edges joining the same two vertices, and so the edges would need to intersect at two points. Each edge however is a great circle arc, which only intersect at antipodal points of $S^2$. In particular, the link $\Sigma$ contains an entire half-great circle whose interior does not intersect any other arcs in $\Sigma$. As $\BC_0$ is comprised of planes, this is impossible unless all the $\widetilde{P}_i$ contain the same line, which would imply $\dim(\BC_0) \geq 1$, a contradiction.\footnote{Alternatively, this can be seen by the maximum principle: simply rotate the half-great circle arc until it touches another point of $\Sigma$ in its interior. By the maximum principle for stationary networks, they would then need to coincide. Repeating inductively shows that $\Sigma$ consists of half-great circles passing through the same two antipodal points.}
	
	With a slight abuse of notation, write $V = \#V$ and $E = \#E$. Also write $F$ for the number of faces on $S^2$ determined by $G$. Euler's polyhedral formula gives
	$$V+F=E+2.$$
	Elementary counting from (a) gives $2E\geq 4V$, and similarly from (b) and (c) we get $2E\geq 4F$ (as every face must be bounded by at least $4$ edges). Combining these inequalities with Euler's formula above gives a contradiction (as $2\geq 0$). Thus $\sfrak\neq n-2$.
	
	We now claim that $\sfrak \neq \ell$ for any $\ell\leq n-2$. Indeed, suppose this were false. Let $\ell$ by the largest integer $<n-2$ for which there is a cone with $\sfrak=\ell$. There are two cases:
	\begin{enumerate}
		\item [(i)] $\dim(P_i\cap P_j\cap P_k) = n-1$ for all $i<j<k$;
		\item [(ii)] $\dim(P_i\cap P_j\cap P_k) = n-2$ for some $i<j<k$.
	\end{enumerate}
	Note that one of (i) or (ii) must occur. If (i) occurs, then as $P_1\cap P_2$ is $(n-1)$-dimensional, we see that $P_j$ must contain this $(n-1)$-dimensional subspace for all $j$. But then this would imply that $\sfrak = n-1$, a contradiction.
	
	If (ii) occurs, then the $(n-2)$-skeleton of the hyperplane arrangement $\cup_i P_i$ is necessarily non-empty (as all lower skeleta have dimension $\leq n-3$). Taking a tangent cone at a point in the $(n-2)$-skeleton produces a new cone of spine dimension $n-2$ which obeys all the assumptions of the lemma, which we just proved was impossible. Thus in either case we arrive at a contradiction. This completes the proof.
\end{proof}

We can now prove Theorem \ref{thm:main-reg-2}.

\begin{proof}[Proof of Theorem \ref{thm:main-reg-2}]
	By Allard's compactness theorem \cite{Allard72}, we know we can pass to a subsequence such that $V_j\rightharpoonup V$ in $B^{n+1}_2(0)$ for some stationary integral varifold $V$ in $B^{n+1}_2(0)$. Let $x\in \sing(V)$. We will analyse the various possibilities by looking at the strata $\widetilde{S}_i:= S_i\setminus S_{i-1}$ which $x$ belongs to, where here $S_i$ denotes the usual Almgren stratification of $\sing(V)$, with $S_{-1}=\emptyset$.
	
	\textbf{Case 1:} $x\in \widetilde{S}_n$. Then there exists a tangent cone $\BC$ with $\BC = q|P|$ for some hyperplane $P$ and $q\in \Z_{\geq 1}$. If $q>Q$, then $x\in W_2$. If $q\leq Q$, then for some $\rho>0$ and all sufficiently large $j$ we can apply Theorem \ref{thm:eps-reg} to $V_j\res B^{n+1}_{2\rho}(x)$, and thus we see that $V\res B^{n+1}_\rho(x)$ is given by a sum of $q$ minimal graphs. In particular, $x\in \mathcal{R}$ and the convergence (of each graph) is smooth on $B^{n+1}_\rho(x)$.
	
	\textbf{Case 2:} $x\in \widetilde{S}_{n-1}$. Then there exists a tangent cone $\BC$ with $\dim(S(\BC))=n-1$, so that $0\in \BC$ is a classical singularity of $\BC$. We must have $\Theta_{\BC}(0)\geq Q$, else one could apply \cite[Theorem C]{MW24} to $V_j$ about $x$ for sufficiently large $j$ to see that $V_j$ has a classical singularity with density $<Q$, contradicting the definition of $\S_Q^\pm$. If $\Theta_{\BC}(0)>Q$, then $x\in W_2$. Otherwise $\Theta_{\BC}(0) = Q$, and again by \cite[Theorem C]{MW24} we have that both $\BC$ must be a sum of $Q$ hyperplanes, and that the convergence of $V_j$ to $V$ must be smoothly as immersions locally about $x$. So $x\in \mathcal{R}$.
	
	\textbf{Case 3:} $x\in \widetilde{S}_i$ for some $i\leq n-2$. Let $\BC$ be a tangent cone at $x$. If $\BC$ contains any classical singularities with density $>Q$, then $x\in W_2$. So, we may assume that all tangent cones only contain classical singularities of density $\leq Q$, which again by \cite[Theorem C]{MW24} must necessarily be immersed classical singularities of density exactly $Q$. We then have two cases:
	\begin{enumerate}
		\item [(a)] $x\in \widetilde{S}_{n-2}$;
		\item [(b)] $x\not\in \widetilde{S}_{n-2}$.
	\end{enumerate}
	In case (a), there is a tangent cone $\BC$ at $x$ with $\dim(S(\BC))=n-2$. By looking at the link of the cross-section of $\BC$, as all classical singularities are immersed, $\BC$ must be the sum of hyperplanes. These hyperplanes must have multiplicity $<Q$, else as all classical singularities have density exactly $Q$, we would need $\BC$ to be a single hyperplane, a contradiction to $x\in\widetilde{S}_{n-2}$. But then as $\BC$ is a varifold limit of some sequence in $\S_Q^\pm$, by Theorem \ref{thm:eps-reg} we have smooth convergence away from $S(\BC)$, and so in particular $\BC$ must be alternating. Hence, by Lemma \ref{lemma:spines}, we get a contradiction. So (a) is impossible.
	
	So we must be in case (b). If a regular component of $\BC$ has multiplicity $>Q$, then $x\in W_2$ (note that in this case, $\BC$ can be free of classical singularities). Otherwise, we may assume that every cone has density $\leq Q$ on its regular part. In particular, we must have from Theorem \ref{thm:eps-reg} that $\BC\res B^{n+1}_2(0)\in\S_Q^\pm$. In particular, $\BC$ is alternating, so from Lemma \ref{lemma:spines} $\BC$ cannot be supported on a union of hyperplanes. Hence $x\in W_1$. This completes the proof.
\end{proof}

\section{The Non-Immersed Case \& Open Questions}\label{sec:rest}

In this section, we give and discuss some open questions for area minimising hypersurfaces mod $p$ left open from the present work, with a particular focus on the general case (where non-immersed classical singularities occur). There are however two questions in the case where all classical singularities are immersed which we mention first. The most straightforward one is:

\textbf{Question 1:} What is the optimal dimension bound on the non-immersed singular set in Theorem \ref{thm:main-reg}?

As mentioned previously, we suspect it should be $n-7$ in line with known singularity models for area minimising hypersurfaces. Nonetheless, as mentioned in Remark \ref{remark:ambient-stable}, the difficulty appears to be in finding a useful form of a stability inequality for an area minimising hypersurface mod $p$. Indeed, the problem comes down to classifying area minimising mod $p$ hypercones in $\R^{n+1}$ for $4\leq n+1\leq 6$ whose links are smoothly immersed. If one could show a version of Simons' classification, namely that such cones must be flat (i.e.~sums of hyperplanes), then one would be done by Lemma \ref{lemma:spines}.

We note that it seems that one must use the area minimising mod $p$ property to improve upon the dimension bound rather than just the stability (on the regular set) and the alternating condition we used in most of this work. Indeed, if one takes $V = \BC_1+\BC_2$, where $\BC_1$ is a hyperplane in $\R^4$ with multiplicity $Q-1$ and $\BC_2$ is the cone over the Clifford torus, then $V\res B^{n+1}_2(0)\in \mathcal{S}_Q^\pm$ and $V$ has a non-immersed singularity at $0$. Hence, the corresponding best dimension bound for elements of $\mathcal{S}_Q^\pm$ is $n-3$, in line with our results. This example is not area minimising mod $p=2Q$.

A slightly more open-ended question is presented by Proposition \ref{prop:relations}. This tells us that the tangent maps at flat singular points are determined entirely by two harmonic functions, which for even $p\geq 6$ restricts the number of unknown functions from $p/2$ to $2$. One could ask whether there is any similar relation between the minimal graphs given in Theorem \ref{thm:eps-reg}:

\textbf{Question 2:} Does an analogue of Proposition \ref{prop:relations} hold for the minimal graphs in Theorem \ref{thm:eps-reg}?

For instance, one can verify that if $f:B^n_1(0)\to \R$ is such that both $f$ and $cf$ are minimal for some constant $c$, then either $f$ is affine or $c\in \{0,\pm 1\}$. We leave this as an exercise for the reader.\footnote{An equivalent formulation is that a function solving both the Laplacian and the $\infty$-Laplacian must be affine.}

Now let us turn to the general (non-immersed) case. Here, \cite[Theorem A]{MW24} still applies to give that the local structure about any flat singular point is given by the graph of a $Q$-valued Lipschitz function $u$ which is in fact generalised-$C^{1,\alpha}$. When there is a non-immersed classical singularity, it is impossible for $u$ to decompose into single-valued functions which are $C^1$. Nonetheless, one can ask to what extent the local structure of $u$ about a flat singular point is determined by its tangent map(s) there. The first natural question to understand for this is:

\textbf{Question 3:} Are tangent maps at (density $Q$) flat singular points in alternating $V\in \S_Q$ unique?

A positive answer to Question 3 coupled with the regularity theorems in \cite{MW24} would imply that the local structure of $V\in \S_Q^\pm$ about a density $Q$ flat singular point is a $C^{1,\alpha}$-perturbation of the (unique) tangent map. In the special case of $2$-dimensional area minimising hypersurfaces mod $p$ in $B^{3}_1(0)$, a positive answer to Question 3 has already been given in \cite{SSS25}.

Furthermore, such tangent maps take on a very special form. Indeed, we saw a special case of this already for the class $\S_Q^\pm$ in Proposition \ref{prop:relations}. For general alternating $V\in \S_Q$, one can argue similarly to the proof of \cite[Theorem 3.7]{DLHMS+22} to show that for any (necessarily homogeneous by \cite{KMW}) tangent map $v:B^n_1(0)\to \A_Q(\R)$, there are homogeneous harmonic polynomials $p,q:B^n_1(0)\to \R$ such that, on each connected component $\Omega$ of $\{q\neq 0\}$, there are constants $c_i(\Omega)$ depending on $\Omega$ such that
$$v|_\Omega = \sum_{i=1}^Q\llbracket p+c_i(\Omega)q\rrbracket.$$
Such a representation tells us that $v$ does not have ``genuine'' branch points. This can be understood in several ways: 
	\begin{itemize}
		\item For any $x\in B^n_1(0)$ and connected component $\Omega$ of the regular set of $v$, in $\Omega\cap B_\rho(x)$ the power series representation of $v$ never contains non-integer powers;
		\item In any connected component $\Omega$ of the regular set of $v$, the smooth functions representing $v_i$ in $\Omega$ extend to smooth functions on $B^n_1(0)$.
	\end{itemize}
Furthermore, as there are finitely many connected components of the regular set of $v$ (by Courant's nodal domain theorem applied to $q$), we see that the graph of $v$ is contained within a finite union of smooth submanifolds with no interior boundary in $\R\times B^n_1(0)$. This would be a natural generalisation of Theorem \ref{thm:main-no-bp} to the general case:

\textbf{Question 4:} If $T$ is an area minimising hypersurface mod $p$ and $x$ is a flat singular point of $T$, must there be a radius $\rho = \rho(x)>0$ and an integer $N=N(x)\in\mathbb{Z}_{\geq 1}$ for which
$$\spt\|T\|\cap B^{n+1}_\rho(x)\subseteq \bigcup^N_{i=1}M_i$$
for some smoothly embedded minimal hypersurfaces $M_i$ with no (interior) boundary in $B_\rho^{n+1}(x)$?

We note that, due to the above representation of tangent maps in the general case, the planar frequency at each flat singular point is $\geq 2$. This means one should be able to upgrade the regularity of the multi-valued function $u$ representing $V$ from generalised-$C^{1,\alpha}$ to generalised-$C^{1,1}$ in a similar way to that seen in Lemma \ref{lemma:regularity} (cf.~Theorem \ref{thm:App-1}).

In fact, there is one special case of Question 4 which the methods of the present paper are already able to answer, namely when all the classical singularities in $\spt\|T\|$ are immersed. We refer the reader to Appendix \ref{app:B} for further discussion.

Next, we have already seen that for general area minimising hypersurfaces mod $p$ it is possible to have singular points with tangent cones of spine dimension $n$ and $n-1$. In Lemma \ref{lemma:spines}, we saw that, when all classical singularities are immersed, it was possible to rule out tangent cones with spine dimension $n-2$. A natural question is whether this is possible to do in general:

\textbf{Question 5:} Suppose $\BC$ is a $2$-dimensional cone in $\R^3$ which is locally area minimising mod $p$. Must $\BC$ then have a spine of dimension $\geq 1$?

We stress that if one tries to argue combinatorially as in our proof of Lemma \ref{lemma:spines}, one first reduces to looking at the support of the link as a combinatorial graph $G$ on $S^2$. One can still argue that $G$ contains no odd cycles or cycles of order $2$. Therefore, if every vertex has degree at least $4$, one arrives at a contradiction in the same way as in Lemma \ref{lemma:spines}. In particular, Question 5 holds under the assumption that the classical singularities in $\spt\smallnorm{\bf{C}}$ are all immersed (cf.~Appendix \ref{app:B}). In the general case, notice that, by construction, each vertex has degree $\geq 3$, and so the only problem is if too many vertices have degree $3$. Vertices of degree $3$ are possible (for $p\neq 4$); when $p\geq 5$, the original link of $\BC$ must have multiplicity on some of the edges which is then lost in the combinatorial graph for this to occur. This gives the following combinatorial problem:

\textbf{Question 5$\mathbf{^\prime}$:} Classify all stationary integral $1$-varifolds on $S^2$ for which each vertex has the same (weighted) degree $Q$ and the associated combinatorial graph contains no odd cycles.

We stress that when $Q=3$, such a classification has already been achieved (see e.g.~\cite[Theorem 1.5(3)]{Tay73}), with the only options being: a great circle; 3 half-great circles between the same antipodal points which meet at $120^\circ$; and the spherical projection of the standard cube. The latter is \emph{not} area minimising mod $3$, and so this shows that Question 5 holds when $p=3$. When $Q=4$ all classical singularities are immersed and so Question 5 follows from Lemma \ref{lemma:spines}.

A positive answer to Question 5 for some $p$ then opens up the question of whether it is possible for a tangent cone to an area minimising hypersurface mod $p$ to have spine dimension between $n-3$ and $n-6$. This is effectively asking the following generalisation of Question 1:

\textbf{Question 6:} Are there any hypercones $\BC$ in $\R^{n+1}$ for $4\leq n+1\leq 6$ which are locally area minimising mod $p$ and have $S(\BC) = \{0\}$?

Appropriate answers to the above questions would then answer the following extension of Theorem \ref{thm:main-reg} to general area minimising hypersurfaces mod $p$:

\textbf{Question 7:} Suppose $T$ is an area minimising hypersurface mod $p$ in $B^{n+1}_2(0)$ with $\del^p T = 0$. Then, is there a relatively closed set $S$ of $B^{n+1}_2(0)$ with $\dim_\H(S)\leq n-7$ such that for each $x\in \spt\|T\|\setminus S$ there is a radius $\rho = \rho(x)>0$ and an integer $N=N(x)\in\mathbb{Z}_{\geq 1}$ for which
$$\spt\|T\|\cap B_\rho^{n+1}(x)\subseteq \bigcup^N_{i=1}M_i$$
for some smoothly embedded minimal hypersurfaces $M_i$ with no (interior) boundary in $B^{n+1}_\rho(x)$?

\appendix

\section{Results for $\S_Q$}\label{app:A}

In this appendix we prove results for the larger class $\S_Q$. These generalise those proved for $\S_Q^\pm$ by not assuming any alternating condition. A key difference is that elements of $\S_Q$ \emph{can} admit branch point singularities, even when all classical singularities are immersed (see e.g.~\cite{SW07}).

For $V\in \S_Q$, write $\mathcal{K}_V^Q$ for the set of density $Q$ flat singular points of $V$. Recall that, by \cite{KMW}, $V$ has a well-defined planar frequency value $\mathcal{N}_V(Z)$ at each $Z\in \mathcal{K}_V^Q$.

Our first result finds the optimal regularity exponent in \cite[Theorem A]{MW24}. We stress that this result allows (density $Q$) \emph{non-immersed} classical singularities.

\begin{theorem}\label{thm:App-1}
Suppose $V\in \S_Q$. Then, $\mathcal{N}_V(Z)\geq 3/2$ for all $Z\in\CK_V^Q$.
\end{theorem}
A direct consequence of Theorem \ref{thm:App-1} is that one can take $\alpha=1/2$ in \cite[Theorem A]{MW24}, which is optimal by the examples seen in \cite{SW07}. We note that this is better regularity than one might initially expect, as in general one expects the best possible regularity for a $Q$-valued $C^{1,\alpha}$ function whose graph is stationary to be $C^{1,1/Q}$.

\begin{proof}
By translating, rotating, and \cite[Theorem A]{MW24}, we may assume without loss of generality that $Z=0\in \CK^Q_V$ and that $V$ is represented by the graph of a generalised-$C^{1,\alpha}$ function $u:B^n_1(0)\to \A_Q(\R)$ with $u(0) = Q\llbracket 0\rrbracket$ and $Du(0) = Q\llbracket 0\rrbracket$. Then by the results in \cite[Section 3]{KMW}, it suffices to show that any (necessarily homogeneous and non-zero) tangent map $v$ to $u$ at $0$ has degree of homogeneity $\beta\geq 3/2$.

From \cite[Corollary 3.4]{KMW}, we know that $\beta\geq 1+\alpha$, where $\alpha = \alpha(n,Q)\in (0,1)$ is as in \cite[Theorem A]{MW24}. In particular, since the average $v_{a}$ is always harmonic (see \cite[$(\FB3)$]{MW24}), if $v_{a}\not\equiv 0$ then $v_a$ is a $\beta$-homogeneous harmonic function, and so $\beta\geq 2$; this completes the proof in this case. Thus, we may assume that $v_a\equiv 0$.

Write $\reg(v)$ for the regular part of $v$, which in the present situation is simply the set of $x$ for which $v^i(x) \neq v^j(x)$ for some $i\neq j$. Arguing as in Proposition \ref{prop:relations} we get that, for each connected component $\Sigma$ of $\reg(v)\cap S^{n-1}$, if $\phi_\Sigma$ is the (unique positive) Dirichlet eigenfunction of $\Delta_{S^{n-1}}$ on $\Sigma$, then there is a constant $a(\Sigma) = (a_1(\Sigma),\dotsc,a_Q(\Sigma))\in \R^Q$ such that
$$v|_\Sigma = \sum^Q_{i=1}\llbracket a_i(\Sigma)\phi_\Sigma\rrbracket.$$
Define now $p:S^{n-1}\to \R$ by
$$p(x) := \begin{cases}
|a(\Sigma)|\phi_\Sigma(x) & \text{if }x\in \Sigma\text{ for each connected component $\Sigma$ of $\reg(v)\cap S^{n-1}$;}\\
0 & \text{otherwise.}
\end{cases}$$
Let $\tilde{p}$ denote the $\beta$-homogeneous extension of $p$ to $\R^n$, and set $f:= \llbracket \tilde{p}\rrbracket + \llbracket -\tilde{p}\rrbracket$. Then, the regularity estimates for general blow-ups in $\S_Q$ obtained in \cite[Theorem 3.12]{MW24} together with Lemma \ref{lem: Balancing Condition} imply that $f$ is a (symmetric) two-valued $C^{1,\alpha}$ harmonic function which is $\beta$-homogeneous. As $v$ is non-vanishing, so is $f$. Thus, \cite[Lemma 4.1]{SW16} implies that $\beta\geq 3/2$, completing the proof.
\end{proof}

We now prove a generalisation of Lemma \ref{lemma:decomposition} to $\S_Q$ when one assumes all (density $Q$) classical singularities of the varifold are immersed.

\begin{theorem}\label{thm:App-2}
Suppose $V\in \S_Q$ is such that:
\begin{enumerate}
\item [\textnormal{(i)}] all density $Q$ classical singularities of $V$ are immersed;
\item [\textnormal{(ii)}] there is a generalised-$C^{1,\alpha}$ function $u:B^n_{1/2}(0)\to \A_Q(\R)$ with $V\res (\R\times B^n_{1/2}(0)) = \mathbf{v}(u)$.
\end{enumerate}
Then exactly one of the following two conclusions must hold: either
\begin{enumerate}
\item [\textnormal{(a)}] $V\in \S_Q^\pm$ (in which case the conclusions of Lemma \ref{lemma:decomposition} hold); or
\item [\textnormal{(b)}] there exist integers $q_1\geq 1$, $q_2\geq 0$, with $2q_1+q_2 = Q$ and:
\begin{itemize}[leftmargin=1.5em]
\item $C^{1,1/2}$ two-valued functions $w_1,\dotsc,w_{q_1}:B^n_{1/2}(0)\to \A_2(\R)$ whose graphs are {stationary,} each of which cannot be written globally as the sum of two $C^1$ single-valued functions;
\item a smooth (single-valued) function $w_*:B^n_{1/2}(0)\to \R$ solving the minimal surface \mbox{equation;}
\end{itemize}
such that
\begin{equation*}
u = q_2\llbracket w_*\rrbracket + \sum^{q_1}_{i=1}\llbracket w_i\rrbracket.
\end{equation*}
Furthermore, the branch sets of each $w_i$ must coincide.
\end{enumerate}
\end{theorem}
Of course, \cite[Theorem A]{MW24} ensures that the graphical assumption in Theorem \ref{thm:App-2}(ii) is guaranteed in a neighbourhood of any density $Q$ flat singular point of $V\in \S_Q$. Thus, an immediate consequence of Theorem \ref{thm:App-2} is an improvement on Theorem \ref{thm:eps-reg} to $V\in \S_Q$ with all classical singularities being immersed, except now parts of the graph may be genuinely two-valued.

\begin{proof}
It suffices to prove that if $V\not\in \S_Q^\pm$, i.e.~that $V$ fails the alternating condition, then conclusion (b) holds.

Let $\widetilde{u}_1\leq\cdots\leq \widetilde{u}_Q$ be Lipschitz functions with $u = \sum^Q_{i=1}\llbracket \widetilde{u}_i\rrbracket$. We may assume $\widetilde{u}_i\not\equiv \widetilde{u}_j$ for some $i\neq j$, else $V$ coincides with a single graph with multiplicity $Q$ and so $V$ would be alternating, a contradiction to our assumption. For $j\in\{1,\dotsc,\lfloor Q/2\rfloor\}$, let $\widehat{w}_j:B^n_{1/2}(0)\to \A_2(\R)$ be the two-valued function defined by
$$\widehat{w}_j:= \llbracket \widetilde{u}_j\rrbracket + \llbracket\widetilde{u}_{Q-j+1}\rrbracket.$$
If additionally $Q$ is odd, then define $\widehat{w}_0:=\widetilde{u}_{(Q+1)/2}$. Let $\mathcal{K}$ denote the (density $Q$) critical set of $u$. Then using the Hopf boundary point lemma (and unique continuation for minimal graphs, in the case where multiplicity may occur), one can argue analogously to the proof of Lemma \ref{lemma:decomposition} to see that, away from $\mathcal{K}$, each $\widehat{w}_j$ is a $C^{1,\alpha}$ two-valued minimal graph and, when $Q$ is odd, $\widehat{w}_0$ is a (single-valued) $C^{1,\alpha}$ minimal graph. Furthermore, by \cite[Theorem A]{KMW} we know that $\dim_\H(\mathcal{K})\leq n-2$, and thus the minimality of each function can be extended across $\mathcal{K}$. We stress that some $\widehat{w}_j$ could decompose globally as two single-valued minimal graphs.

We claim that $V$ must contain a branch point. Indeed, if $V$ were free of branch points, then $u$ would decompose globally as a sum of $Q$ single-valued minimal graphs, say $u_1,\dotsc,u_Q$. Let $i,j$ be indices for which $u_i\not\equiv u_j$ (such indices must exist, as otherwise $u = Q\llbracket u_1\rrbracket$ and so $V$ would obey the alternating condition). Let $v:= u_i-u_j$. Then $v\not\equiv 0$ is analytic, and moreover as $\{v=0\} = \CC_u\cup\CK$, we have $|v|>0$ on $\reg(u)$. Notice that $v$ must change sign across any two components of $\reg(u)$ whose boundaries contain a common point in $\CC_u$; indeed, $v$ solves a divergence-form elliptic equation with analytic coefficients, and so the Hopf boundary point lemma applies. Thus, as we are in codimension one, we can define an orientation on $\reg(V)$ via a choice of unit normal on $\reg(V)$, and we can simply choose $\text{sign}(v)\nu$, where $\nu$ is (say) the upward-pointing unit normal to the graph of $u$. By construction, $V$ would therefore be alternating, a contradiction.

Thus, $V$ contains a branch point. Consequently, one of the two-valued functions $\widehat{w}_i$ must be genuinely two-valued. We therefore now know that there exist integers $q_1\geq 1$ and $q_2\geq 0$ for which there are $C^{1,1/2}$ two-valued functions $w_1,\dotsc,w_{q_1}$ with stationary graphs (the regularity coming from Theorem \ref{thm:App-1}) and which cannot be written globally as the sum of two $C^1$ single-valued functions, as well as (smooth) single-valued minimal graphs $v_1,\dotsc,v_{q_2}$ for which
$$u = \sum^{q_1}_{i=1}\llbracket w_i\rrbracket + \sum^{q_2}_{i=1}\llbracket v_i\rrbracket.$$
Clearly we need $2q_1+q_2=Q$. We now claim that $v_i\equiv v_j$ for all $i\neq j$. Of course if $q_2=0$ or $q_2=1$ then we are done, so let us suppose $q_2\geq 2$. But then notice that our argument showing that $V$ must contain a branch point also applies to show that $v_i\equiv v_j$ for all $i,j$ (else consider signs of $v:=v_i-v_j$), and thus if we write $w_* := v_1$ we have $v_i \equiv w_*$ for all $i\in\{1,\dotsc,q_2\}$.

Finally, notice that the branch sets of each $w_i$ must coincide, as otherwise locally about a branch point of, say, $w_1$, which is not a branch point of, say, $w_2$, writing locally $w_2$ as the sum of two \emph{distinct} minimal graphs, one could run the above argument with these two minimal graphs to see that $V$ is alternating in this region, contradicting $w_1$ being branched (by Theorem \ref{thm:eps-reg}).
\end{proof}

\begin{remark}
    We make a few remarks concerning Theorem \ref{thm:App-2}.
\begin{enumerate}
\item [(1)] We used \cite{KMW} to prove the dimension bound $\dim_\H(\mathcal{K})\leq n-2$ in the proof of Theorem \ref{thm:App-2}. It seems plausible that one might be able to instead prove this bound in a more elementary manner using arguments similar to those seen in \cite{SW16}. The dimension bound in \cite{SW16} does not appear to  be directly applicable as one does not know stationarity of the individual two-valued functions in the proof at that stage (recall the key point in the proof of Theorem \ref{thm:eps-reg} was that the dimension bound could be obtained via elementary arguments).

\item [(2)] The proof of Theorem \ref{thm:App-2} also shows that, in the case that $V\in \S_Q$ has all classical singularities immersed, having no (density $Q$) branch points is equivalent to the alternating condition holding (in the region $\{\Theta_V<Q+1/2\}$). Thus, a density $Q$ flat singular point $X$ of $V$ is a branch point if and only if either the alternating condition fails in a neighbourhood of $X$, or $X$ is a limit of non-immersed classical singularities.

\item [(3)] In case (b) of Theorem \ref{thm:App-2}, one also expects relations between the functions $w_i$ and $w_*$, enforced by the fact that all classical singularities have density $Q$. For instance, for any $1\leq i\leq q_1$, if one writes $w_i = \llbracket w_i^{(1)}\rrbracket + \llbracket w_i^{(2)}\rrbracket$ where $w_i^{(1)}\leq w_i^{(2)}$, then $w_i^{(1)}\leq w_*\leq w^{(2)}_i$. Indeed, if this were to fail, then one could consider $\widetilde{w}:=\mb{w_i^{(1)}-w_*}+\mb{w_i^{(2)}-w_*}$ and note that both components of $\widetilde{w}$ have the same constant sign on each connected component of $\reg(u)$. Furthermore, the Hopf boundary point lemma implies that this sign must flip across classical singularities, and we could thus argue analogously to the proof of Theorem \ref{thm:App-2} and define an orientation on $\reg(V)$ according to this sign, which would contradict the assumption that $V$ is not alternating.

\item [(4)] A consequence of Theorem \ref{thm:App-2} is that if $u\in C^{1}(B^n_{1}(0);\A_Q(\R))$ has a stationary graph and all classical singularities having density $Q$, then in fact $u\in C^{1,1/2}(B^n_{1/2}(0);\A_Q(\R))$ and $u$ admits a decomposition as in one of the conclusions of Theorem \ref{thm:App-2}. Indeed, this follows because $\mathbf{v}(u)\in \S_Q$ and all (density $Q$) classical singularities of $\mathbf{v}(u)$ are immersed.  
\end{enumerate}
\end{remark}
Combining Theorem \ref{thm:App-2} with \cite[Theorem A]{KW21} gives an asymptotic expansion for $V\in \S_Q$ at $\H^{n-2}$-a.e.~density $Q$ immersed branch point of $V$. Furthermore, combining Theorem \ref{thm:eps-reg} and Theorem \ref{thm:App-2} with \cite[Theorem C]{KW21} gives the following structure for the set of density $Q$ immersed branch points of $V$:
\begin{corollary}\label{cor:App-4}
Let $V\in\S_Q$ be $n$-dimensional. Suppose all density $Q$ classical singularities of $V$ are immersed. If $n=2$, then the density $Q$ flat singular points of $V$ are isolated. If $n\geq 3$, then for all $\Omega\Subset B^{n+1}_2(0)$, the set $\CK^Q_V\cap\Omega$ is either empty, or has positive $\H^{n-2}$-measure, and decomposes as a finite disjoint union of locally compact sets, each of which is locally $(n-2)$-rectifiable.
\end{corollary}
Recently, \cite{SSS25} and \cite{KMW} have independently established that flat singular points of $2$-dimensional area minimising hypersurfaces mod $p$ are isolated.

We conclude this appendix by showing a certain `balancing' condition holds for coarse blow-ups $v\in\mathfrak{B}_Q$ of varifolds in $\S_Q$ (notation as in \cite{MW24}). 
This was used in the proof of Theorem \ref{thm:App-1}.

\begin{lem}\label{lem: Balancing Condition}
Suppose $v\in\FB_Q$. Then, $|Dv|:B^n_1(0)\to \R$ is continuous, where $|Dv|^2 = \sum_{i=1}^Q|Dv^i|^2.$
\end{lem}
\begin{proof}
Fix $v\in\mathfrak{B}_Q$. By \cite[Theorem 3.12]{MW24}, $|Dv|$ is clearly continuous away from $\CC_v$, the set of classical singularities of $v$, so it suffices to check the continuity along $\CC_v$. Fix $z\in \CC_v$. As the derivatives of $v$ in directions parallel to the classical singularities agree, it suffices to prove that the norm of the (one-sided) normal derivatives agree on either side of the classical singularity at $z$. 

By taking a tangent map to $v$ at $z$ and using \cite[Theorem 3.12, $(\mathfrak{B}6)$, $(\mathfrak{B}5)$]{MW24}, it suffices to prove this in the case where $v\in\mathfrak{B}_Q$ takes the form
$$v(x) = \begin{cases}
    \sum^Q_{i=1}\llbracket A_i x^1\rrbracket & \text{if }x^1\geq 0;\\
    \sum^Q_{i=1}\llbracket B_i x^1\rrbracket & \text{if }x^1<0;
\end{cases}$$
for some constants $A_i,B_j\in \R$. We are therefore reduced to showing $|A| = |B|$, i.e. $\sum_{i=1}^Q |A_i|^2 = \sum_{i=1}^Q|B_i|^2$. By \cite[Lemma 3.7]{MW24}, we know that $v$ satisfies the squeeze identity, and so
$$\int\sum^Q_{\alpha=1}\sum^n_{i,j=1}\big(|Dv^\alpha|^2\delta_{ij} - 2D_iv^\alpha\cdot Dv_j^\alpha\big)D_i\zeta^j = 0 \qquad \text{for all }\zeta\in C^\infty_c(\R^n;\R^n).$$
Using the specific form of $v$, this gives
$$|A|^2\int_{\{x^1>0\}}(\div(\zeta)-2D_1\zeta^1) + |B|^2\int_{\{x^1<0\}}(\div(\zeta)-2D_1\zeta^1) = 0.$$
Choose $\zeta^j\equiv 0$ for all $j>1$ and then integrate by parts to get $(|A|^2-|B|^2)\int_{\{x^1=0\}}\zeta^1 = 0$  for all $\zeta^1\in C^\infty_c(\R^n;\R)$.
This implies $|A|=|B|$, which completes the proof.
\end{proof}

\section{A Note on the Non-Immersed Case}\label{app:B}

In this appendix we detail a generalisation of the results in the paper which allows $V\in \S_Q$ to have certain non-immersed classical singularities. The generalisation here is, instead of assuming all density $Q$ classical singularities of $V$ are immersed, we only assume that locally about each density $Q$ classical singularity of $V$, the \emph{support} $\spt\|V\|$ is smoothly immersed. Equivalently, if $\BC$ is a tangent cone at a density $Q$ classical singularity of $V$, then $\spt\|\BC\| = \cup_i P_i$, for (distinct) hyperplanes $P_i$.

Notice that this is a strictly weaker assumption than assuming all (density $Q$) classical singularities in $V$ are immersed, as it is possible for the cone $\BC$ to not be a sum of hyperplanes (as a varifold), whilst its support $\spt\|\BC\|$ is a union of hyperplanes. For instance, let $\BC = \BC_1+2\BC_2$, where $\BC_1$ is the standard triple junction in $\R^2$, and $\BC_2:= \Gamma_\#\BC_1$, where $\Gamma\in O(2)$ is the rotation by angle $\pi$. Notice that $\BC$ is also the sum of $\BC_1$ along with a copy of each line extending the rays in $\BC_1$.

Under this weaker assumption, we still have a version of Theorem \ref{thm:eps-reg} in this setting, except now for $\spt\|V\|$ rather than $V$ itself. This is the following:

\begin{theorem}\label{thm:eps-reg weaker assumptions}
There exists $\eps=\eps(n,Q)\in(0,1)$ such that the following is true. Suppose $V\in\S_Q$ is alternating and that the classical singularities of $\spt\smallnorm{V}\cap\{\Theta_V<Q+1/2\}$ are all immersed. If $(2^n\w_n)^{-1}\|V\|(B^{n+1}_2(0))<Q+\frac{1}{2}$ and
$$\dist_\H(\spt\|V\|\cap (\R\times B^n_1(0)), \{0\}\times B^n_1(0))<\eps,$$
then for some $q\in\{1,\dots,Q\}$ there exist (distinct) smooth functions $u_1,\dots,u_q:B^n_{1/2}(0)\to\R$ solving the minimal surface equation for which
$$
\spt\smallnorm{V}\cap(\R\times B^n_{1/2}(0))=\cup_{i=1}^q\graph(u_i)\qquad \text{and} \qquad \smallnorm{u_i}_{C^\ell}\leq C(n,Q,\ell)\hat{E}_V\quad\text{ for all }\ell\geq0.
$$
\end{theorem}

In particular, Theorem \ref{thm:eps-reg weaker assumptions} provides a positive answer to Question 4 in the case when the classical singularities of $\spt\smallnorm{T}$ are all immersed. It also illustrates a key problem in the general case: if one takes the multi-valued function provided by \cite[Theorem A]{MW24}, and looks at a region where it is described locally by single-valued minimal graphs (which exists as the regular part is always non-empty by Allard regularity) one wishes to know if these functions extend to single-valued minimal functions on the whole ball. In the cases discussed in the present paper, this is possible to do precisely using the immersed structure of classical singularities.

\begin{proof}
The difficulty of the proof in comparison to the earlier results of the paper is that, if $V_{(1)}$ denotes the multiplicity one varifold associated to $\spt\|V\|$, then whilst $V_{(1)}$ has stable regular part and is alternating, it is not clear whether $V_{(1)}$ is stationary, or whether all classical singularities in $V_{(1)}$ have a common density. In fact, these properties will follow, once we know the dimension bound on the (density $Q$) branch set of $V$. For this dimension bound, one could appeal to \cite[Theorem A]{KMW} (which, we recall, uses a center manifold construction). Instead, we describe below how our simple PDE methods generalise under this weaker assumption.

We adopt the notation of the proof of Lemma \ref{lemma:decomposition}, in particular $u$ is the $Q$-valued function representing $V$. The immersed assumption on the classical singularities of $\graph(u)$ ensures that $w_1$ and $w_Q$, namely the ``top'' and ``bottom'' sheets, are still well-defined (even if all classical singularities in $V_{(1)}$ do not have a common density). Moreover, the same argument as in Lemma \ref{lemma:decomposition} shows that $w_1$ and $w_Q$ are $C^{1,\alpha}$ and minimal away from their touching set (which we still denote by $\CK$).
We now recycle the proofs of Lemmas \ref{lemma:blow-ups}, \ref{lemma:lower-bound}, and \ref{lemma:regularity}, only focusing on $w_1$ and $w_Q$, namely:
\begin{itemize}
    \item In Lemma \ref{lemma:blow-ups} we can always prove that $w_1$ and $w_Q$ linearise to harmonic functions when blowing-up. Indeed, if $w_1\not\equiv w_Q$, one can run the same proof as in Lemma \ref{lemma:blow-ups}. If $w_1\equiv w_Q$, then in fact the whole coarse blow-up must coincide with a single function (as $w_1$ and $w_Q$ are the top and bottom sheets), and thus be harmonic (by \cite[$(\mathfrak{B}3)$]{MW24}).
    \item The proof of Lemma \ref{lemma:lower-bound} is the same, noting that since the tangent map is always non-zero, the blow-ups $v_1,v_Q$ of $w_1$, $w_Q$ (the top and bottom sheets) must be non-zero, as again if they were to both be zero then the whole blow-up would need to coincide with a single function, which would necessarily be zero giving a contradiction. As one of $v_1,v_Q$ is non-zero this gives the same frequency lower bound.
    \item The proof of Lemma \ref{lemma:regularity} is identical, just applied to the top and bottom sheets $w_1,w_Q$.
\end{itemize}
The proof of the dimension bound $\dim_\H(\CK)\leq n-2$ then follows in an identical way to that seen in the proof of Theorem \ref{thm:eps-reg}, just working with the top and bottom sheets.

In particular, it now follows that $V_{(1)}$ is stationary (as it is stationary away from its branch set, which we now know to be small enough to excise). Furthermore, $B^n_{1/2}(0)\setminus\CK$ is connected, and thus the assumption on the classical singularities of $\spt\|V\|$ implies that the number of distinct elements of the set $\{u_1(x),\dotsc,u_Q(x)\}$ is constant for $x\in \reg(u)$. Let $Q^\prime\in \{1,\dotsc,Q\}$ denote this number. Then by the above, $V_{(1)}\in\S_{Q^\prime}^\pm$ and thus the result follows from Theorem \ref{thm:eps-reg}.
\end{proof}

Notice that Theorem \ref{thm:eps-reg weaker assumptions} gives the following. Fix $Q\in \{2,3,\dotsc\}$, and suppose $V\in \S_Q$ is alternating and that locally about each density $Q$ classical singularity of $V$, the support $\spt\|V\|$ is smoothly immersed. Then, locally about any density $Q$ flat singular point $x\in \sing(V)$, there exists $\rho>0$ such that $(\eta_{x,\rho})_\#V_{(1)}\in \S_{Q^\prime}^\pm$ for some $Q^\prime\in\{2,\dotsc,Q\}$ (depending on $x$). In particular, we immediately see that $x$ is not a branch point of $\spt\|V\|$. Furthermore, one can also apply Lemma \ref{lemma:spines} to show that $\spt\|V\|\cap\{\Theta_V<Q+1/2\}$ is a two-sided smooth immersion away from a closed set of codimension $\geq 3$. These observations provide generalisations of Theorems \ref{thm:main-no-bp}, \ref{thm:main-reg}, \ref{thm:main-compactness}, and \ref{thm:main-reg-2} to this setting.

It is worth noting that if $V$ is the natural varifold associated to an area minimising hypersurface mod $p$, then it is not at all clear whether $V_{(1)}$ is itself, locally about a flat singular point, the natural varifold associated to an area minimising hypersurface mod $p^\prime$, for some $p^\prime\in\{1,\dots,p\}$. The above exhibits that, in this setting, the variational and structural properties of the hypersurface allow one to bypass these complexities. 

Finally, one can more generally use the dimension bound in \cite[Theorem A]{KMW} to see that if $Q\in\{2,3,\dots\}$ and $V\in\S_Q$ such that $\spt\|V\|$ is smoothly immersed locally about each density $Q$ classical singularity of $V$, then, for any density $Q$ flat singular point $x\in\sing(V)$, there exists $Q^\prime\in\{2,\dots,Q\}$ (depending on $x\in\sing(V)$) such that some rescaling of $V_{(1)}$ about $x$ is in $\S_{Q^\prime}$. This gives a generalisation of Theorem \ref{thm:App-2} to this setting. In particular, the conclusions of Corollary \ref{cor:App-4} remain unchanged under this weaker assumption.

\newcommand{\etalchar}[1]{$^{#1}$}

\end{document}